\documentclass[a4paper]{article} 
\RequirePackage{fullpage}
\usepackage[usenames,dvipsnames]{xcolor}

\usepackage{booktabs}
\usepackage{subcaption}
\usepackage{enumerate}
\usepackage{charter} 
\usepackage{amsthm, amsmath, amssymb} 
\usepackage[euler-digits,small]{eulervm}
\usepackage{tikz}
\usepackage[colorlinks=true,linkcolor=black,citecolor=BlueViolet,urlcolor=Maroon]{hyperref}

\tikzstyle{filled vertex}=[fill=cyan, circle, thin, draw, inner sep=0pt, minimum width=3pt]
\tikzstyle{solution vertex}=[draw, thin, minimum width=3pt]
\tikzstyle{normal edge}=[thick]
\tikzstyle{potential edge}=[ForestGreen]
\tikzstyle{forbidden edge}=[line width=0.20mm, RubineRed, dashed]

\newtheorem{theorem}{Theorem}
\newtheorem{proposition}[theorem]{Proposition}
\newtheorem{lemma}[theorem]{Lemma}
\newtheorem{corollary}[theorem]{Corollary}

\title{Complementation in T-perfect Graphs}
\author{Yixin Cao\thanks{Department of Computing, Hong Kong Polytechnic University, Hong Kong, China. {Supported in part by the Hong Kong Research Grants Council (RGC) under grants 15201317 and 15226116, and the National Natural Science Foundation of China (NSFC) under grant 61972330.}  {\tt yixin.cao@polyu.edu.hk, shenghua.wang@connect.polyu.hk}.} 
  \and Shenghua Wang\footnotemark[1]
}
\date{}

\begin{document}
\maketitle

\begin{abstract}
  Inspired by applications of perfect graphs in combinatorial optimization, Chv{\'a}tal defined t-perfect graphs in 1970s.  The long efforts of characterizing t-perfect graphs started immediately, but embarrassingly, even a working conjecture on it is still missing after nearly 50 years.  Unlike perfect graphs, t-perfect graphs are not closed under substitution or complementation.  A full characterization of t-perfection with respect to substitution has been obtained by Benchetrit in his Ph.D.~thesis.  Through the present work we attempt to understand t-perfection with respect to complementation.  In particular, we show that there are only five pairs of graphs such that both the graphs and their complements are minimally t-imperfect.
\end{abstract}

\section{Introduction}

Partly motivated by Shannon's work on communication theory, Berge proposed the concept of perfect graphs, and the two perfect graph conjectures \cite{berge-63-perfect-graphs}, both settled now.
Chv{\'{a}}tal \cite{chvatal-75-graph-polytopes} and Padberg \cite{padberg-74-perfect-matrices} independently showed that the independent set polytope of a perfect graph (the convex hull of incidence vectors of independent sets of the graph) is determined by non-negativity and clique inequalities,
and this is part of efforts trying to characterize the {independent set polytope} of a graph~\cite{padberg-74-perfect-matrices, nemhauser-74-vertex-packing-polyhedra, nemhauser-75, chvatal-75-graph-polytopes}.
Chv{\'{a}}tal~\cite{chvatal-75-graph-polytopes} went further to propose a class of graphs directly defined by the properties of their independent set polytopes.  A graph is \textit{t-perfect} if its independent set polytope can be fully described by non-negativity, edge, and odd-cycle inequalities. 
The class of t-perfect graphs and the class of perfect graphs are incomparable: $C_5$ is t-perfect but not perfect, while $K_4$ is perfect but not t-perfect.
Similar as perfect graphs, the maximum independent set problem can be solved in polynomial time in t-perfect graphs~\cite{grotschel-86-relaxations-vertex-packing}; see also Eisenbrand et al.~\cite{eisenbrand-03-independent-set-t-perfect}.

The progress toward understanding t-perfection has been embarrassingly slow.  While the original paper of Berge~\cite{berge-60-perfect-conjecture} on perfect graphs already contains several important subclasses of perfect graphs, thus far, only few graph classes are known to be t-perfect.  Because of the absence of odd cycles, bipartite graphs are trivial examples, and this can be generalized to {almost bipartite graphs}~\cite{fonlupt-82-h-perfectness}.  Another class is the series-parallel graphs~\cite{boulala-79-series-parallel}.
Extending these two classes, Gerards~\cite{gerards-89-stable-sets} showed that any graph contains no \textit{odd-$K_{4}$} (a subdivision of $K_{4}$ in which every triangle of $K_{4}$ becomes an odd cycle) is t-perfect.

The strong perfect graph theorem states that a graph $G$ is perfect if and only if $G$ does not contain any odd hole (an induced cycle of length at least four) or its complement~\cite{chudnovsky-06-strong-perfect-graphs-theorem}.  In other words, the {minimally imperfect} graphs (imperfect graphs whose proper induced subgraphs are all perfect) are  odd holes and their complements.
Naturally, we would like to know minimally t-imperfect graphs, whose definition is more technical and is left to the next section.
However, even a working conjecture on them is still missing.
Full success has only been achieved on special graphs, e.g., claw-free graphs~\cite{bruhn-12-claw-free-t-perfect} and $P_{5}$-free graphs~\cite{bruhn-17-t-perfection-p5-free}.
In summary, known minimally t-imperfect graphs include $(3,3)$-partitionable graphs~\cite{chvatal-79-combinatorial-designs,bruhn-10-t-perfection-claw-free}, odd wheels~\cite{schrijver-03}, even M\"{o}bius ladders~\cite{shepherd-95-packing}, and the complements of some cycle powers ($\overline{C_{7}}$, $\overline{C_{13}^{3}}$, $\overline{C_{13}^{4}}$, $\overline{C_{19}^{7}}$); see Figures~\ref{fig:minimally t-imperfect} and~\ref{fig: D graphs}.  Note that Figure~\ref{fig: D graphs}(f) is precisely $\overline{C_{10}^{2}}$, while all even M\"{o}bius ladders are also complements of cycle powers.

\begin{figure}[h]
  \centering
  \tikzset{every path/.style={normal edge}, every node/.style={filled vertex}}
  \begin{subfigure}[b]{4.5cm}
    \centering
    \begin{tikzpicture}
      \foreach \i in {1,..., 3} {
          \draw ({90 - \i * (360 / 3)}:1) -- ({210 - \i * (360 / 3)}:1);
        }
      \foreach \i in {1,..., 3} {
          \draw ({90 - \i * (360 / 3)}:1) -- (0:0);
        }
      \foreach \i in {1,..., 3} {
          \node[filled vertex] at ({90 - \i * (360 / 3)}:1) {};
        }
      \node[filled vertex] at (0:0) {};
    \end{tikzpicture}
    \,
    \begin{tikzpicture}
      \foreach \i in {1,..., 5} {
          \draw ({18 - \i * (360 / 5)}:1) -- ({90 - \i * (360 / 5)}:1);
        }
      \foreach \i in {1,..., 5} {
          \draw ({18 - \i * (360 / 5)}:1) -- (0:0);
        }
      \foreach \i in {1,..., 5} {
          \node[filled vertex] at ({18 - \i * (360 / 5)}:1) {};
        }
      \node[filled vertex] at (0:0) {};
    \end{tikzpicture}
    \caption {$W_{2k+1}$ (odd wheels)}
  \end{subfigure}
  \begin{subfigure}[b]{5cm}
    \centering
    \begin{tikzpicture}
      \foreach \i in {1,..., 8} {
          \draw ({(90 - 360 / 8) - \i * (360 / 8)}:1) -- ({90 - \i * (360 / 8)}:1);
        }
      \foreach \i in {1,..., 4} {
          \draw ({90 - \i * (360 / 8)}:1) -- ({(90 - 360 / 8) - (\i + 3) * (360 / 8)}:1);
        }
      \foreach \i in {1,..., 8} {
          \node[filled vertex] at ({(90 - 360 / 8) - \i * (360 / 8)}:1) {};
        }
    \end{tikzpicture}
    \,
    \begin{tikzpicture}[scale = .68]
      \foreach \i in {1,..., 5} {
          \draw ({120 - \i * (360 / 6)}:1) -- ({60 - \i * (360 / 6)}:1);
          \draw ({120 - \i * (360 / 6)}:1.5) -- ({60 - \i * (360 / 6)}:1.5);
        }
      \foreach \i in {1,..., 6} {
          \node[filled vertex] (v\i) at ({\i * (360 / 6)}:1) {};
          \node[filled vertex] (u\i) at ({\i * (360 / 6)}:1.5) {};
          \draw (v\i) -- (u\i);
        }
      \draw (v1) -- (u2) (v2) -- (u1);
    \end{tikzpicture}
    \caption{$M_{2 k}$ (even M\"{o}bius ladders)}
  \end{subfigure}
  \begin{subfigure}[b]{5cm}
    \centering
    \begin{tikzpicture}
      \foreach \i in {1,..., 7} {
          \draw ({(90 - 360 / 7) - \i * (360 / 7)}:1) -- ({90 - \i * (360 / 7)}:1);
        }
      \foreach \i in {1,..., 7} {
          \draw ({90 - \i * (360 / 7)}:1) -- ({(90 - 360 / 7) - (\i + 1) * (360 / 7)}:1);
        }
      \foreach \i in {1,..., 7} {
          \node[filled vertex] at ({(90 - 360 / 7) - \i * (360 / 7)}:1) {};
        }
    \end{tikzpicture}
    \,
    \begin{tikzpicture}
      \foreach \i in {1,..., 10} {
          \draw ({\i * (360 / 10)}:1) -- ({36 + \i * (360 / 10)}:1);
        }
      \foreach \i in {1,..., 5} {
          \draw ({\i * (360 / 5)}:1) -- ({72 + \i * (360 / 5)}:1);
          \draw ({\i * (360 / 5)}:1) -- ({180 + \i * (360 / 5)}:1);
        }
      \foreach \i in {1,..., 5} {
          \draw ({36 + \i * (360 / 5)}:1) -- ({-36 + \i * (360 / 5)}:1);
        }
      \foreach \i in {1,..., 5} {
          \node (v\i) at ({144 + \i * (720 / 5)}:1) {};
        }
      \foreach \i in {1,..., 5} {
          \node (v\i) at ({-36 + \i * (720 / 5)}:1) {};
        }
    \end{tikzpicture}
    \caption{complements of cycle powers }
  \end{subfigure}
  \caption{Some minimally t-imperfect graphs.  Shown here are $W_3$, $W_5$, $M_{4}$, $M_{6}$, $\overline{C_{7}}$, and $\overline{C_{10}^{2}}$.}
  \label{fig:minimally t-imperfect}
\end{figure}
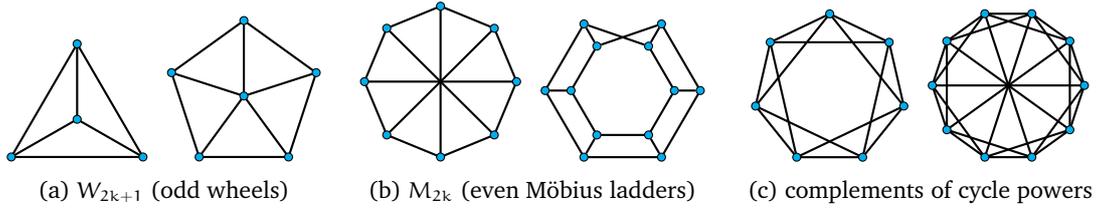

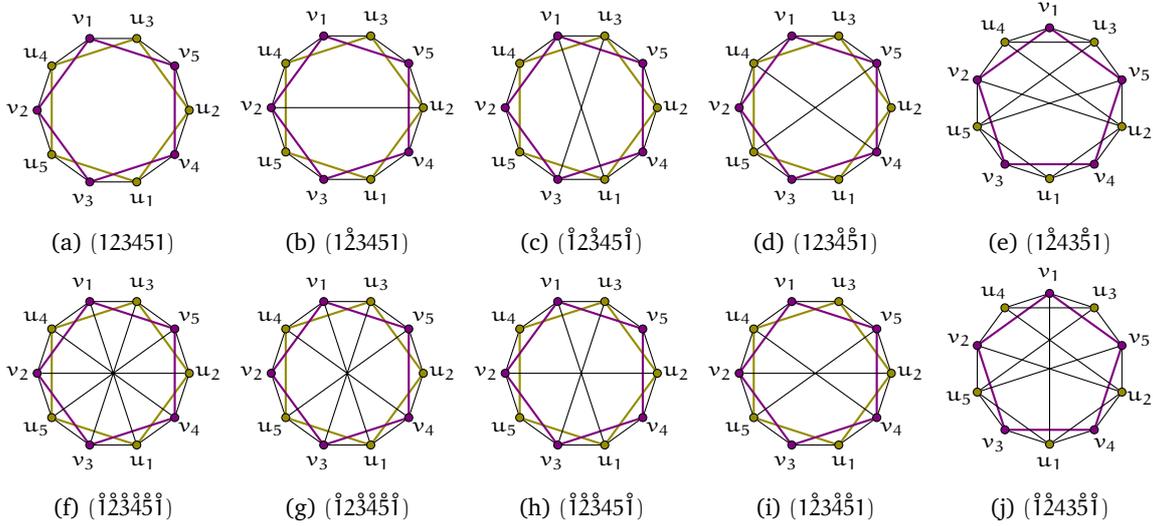
\begin{figure}[h]
  \centering
  \footnotesize
  \begin{subfigure}[b]{3cm}
    \centering
    \begin{tikzpicture}[object/.style={thin,double,<->}]
      \foreach \i in {1,..., 10} {
          \draw[thin] ({\i * (360 / 10)}:1) -- ({36 + \i * (360 / 10)}:1);
        }
      \foreach \i in {1,..., 5} {
          \draw[normal edge, olive] ({\i * (360 / 5)}:1) -- ({72 + \i * (360 / 5)}:1);
        }
      \foreach \i in {1,..., 5} {
          \draw[normal edge, violet] ({36 + \i * (360 / 5)}:1) -- ({-36 + \i * (360 / 5)}:1);
        }
      \foreach \i in {1,..., 5} {
          \node[filled vertex, fill = violet] (v\i) at ({36 + \i * (360 / 5)}:1) {};
          \node at ({36 + \i * (360 / 5)}:1.25) {$v_{\i}$};
        }
      \foreach \i in {1,..., 5} {
          \node[filled vertex, fill = olive] (u\i) at ({216 + \i * (360 / 5)}:1) {};
          \node at ({216 + \i * (360 / 5)}:1.25) {$u_{\i}$};
        }
    \end{tikzpicture}
    \caption{$(123451)$}
  \end{subfigure}
  \begin{subfigure}[b]{3cm}
    \centering
    \begin{tikzpicture}[object/.style={thin,double,<->}]
      \foreach \i in {1,..., 10} {
          \draw[thin] ({\i * (360 / 10)}:1) -- ({36 + \i * (360 / 10)}:1);
        }
      \foreach \i in {1,..., 5} {
          \draw[normal edge, olive] ({\i * (360 / 5)}:1) -- ({72 + \i * (360 / 5)}:1);
        }
      \foreach \i in {1,..., 5} {
          \draw[normal edge, violet] ({36 + \i * (360 / 5)}:1) -- ({-36 + \i * (360 / 5)}:1);
        }
      \foreach \i in {1,..., 5} {
          \node[filled vertex, fill = violet] (v\i) at ({36 + \i * (360 / 5)}:1) {};
          \node at ({36 + \i * (360 / 5)}:1.25) {$v_{\i}$};
        }
      \foreach \i in {1,..., 5} {
          \node[filled vertex, fill = olive] (u\i) at ({216 + \i * (360 / 5)}:1) {};
          \node at ({216 + \i * (360 / 5)}:1.25) {$u_{\i}$};
        }
      \draw (u2) -- (v2);
    \end{tikzpicture}
    \caption{$(1\mathring{2}3451)$}
  \end{subfigure}
  \begin{subfigure}[b]{3cm}
    \centering
    \begin{tikzpicture}[object/.style={thin,double,<->}]
      \foreach \i in {1,..., 10} {
          \draw[thin] ({\i * (360 / 10)}:1) -- ({36 + \i * (360 / 10)}:1);
        }
      \foreach \i in {1,..., 5} {
          \draw[normal edge, olive] ({\i * (360 / 5)}:1) -- ({72 + \i * (360 / 5)}:1);
        }
      \foreach \i in {1,..., 5} {
          \draw[normal edge, violet] ({36 + \i * (360 / 5)}:1) -- ({-36 + \i * (360 / 5)}:1);
        }
      \foreach \i in {1,..., 5} {
          \node[filled vertex, fill = violet] (v\i) at ({36 + \i * (360 / 5)}:1) {};
          \node at ({36 + \i * (360 / 5)}:1.25) {$v_{\i}$};
        }
      \foreach \i in {1,..., 5} {
          \node[filled vertex, fill = olive] (u\i) at ({216 + \i * (360 / 5)}:1) {};
          \node at ({216 + \i * (360 / 5)}:1.25) {$u_{\i}$};
        }
      \draw (u1) -- (v1) (u3) -- (v3);
    \end{tikzpicture}
    \caption{$(\mathring{1}2\mathring{3}45\mathring{1})$}
  \end{subfigure}
  \begin{subfigure}[b]{3cm}
    \centering
    \begin{tikzpicture}[object/.style={thin,double,<->}]
      \foreach \i in {1,..., 10} {
          \draw[thin] ({\i * (360 / 10)}:1) -- ({36 + \i * (360 / 10)}:1);
        }
      \foreach \i in {1,..., 5} {
          \draw[normal edge, olive] ({\i * (360 / 5)}:1) -- ({72 + \i * (360 / 5)}:1);
        }
      \foreach \i in {1,..., 5} {
          \draw[normal edge, violet] ({36 + \i * (360 / 5)}:1) -- ({-36 + \i * (360 / 5)}:1);
        }
      \foreach \i in {1,..., 5} {
          \node[filled vertex, fill = violet] (v\i) at ({36 + \i * (360 / 5)}:1) {};
          \node at ({36 + \i * (360 / 5)}:1.25) {$v_{\i}$};
        }
      \foreach \i in {1,..., 5} {
          \node[filled vertex, fill = olive] (u\i) at ({216 + \i * (360 / 5)}:1) {};
          \node at ({216 + \i * (360 / 5)}:1.25) {$u_{\i}$};
        }
      \draw (u4) -- (v4) (u5) -- (v5);
    \end{tikzpicture}
    \caption{$(123\mathring{4}\mathring{5}1)$}
  \end{subfigure}
  \begin{subfigure}[b]{3cm}
    \centering
    \begin{tikzpicture}[object/.style={thin,double,<->}]
      \foreach \i in {1,..., 10} {
          \draw[thin] ({\i * (360 / 10) - 18}:1) -- ({\i * (360 / 10) + 18}:1);
        }
      \foreach \i in {1,..., 5} {
          \draw[normal edge, violet] ({\i * (360 / 5) - 54}:1) -- ({\i * (360 / 5) + 18}:1);
        }
      \foreach \i in {1,..., 5} {
          \node[filled vertex, fill = violet] (v\i) at ({\i * (360 / 5) + 18}:1) {};
          \node at ({\i * (360 / 5) + 18}:1.25) {$v_{\i}$};
        }
      \foreach \i in {1,..., 5} {
          \node[filled vertex, fill = olive] (u\i) at ({\i * (360 / 5) - 162}:1) {};
          \node at ({\i * (360 / 5) - 162}:1.25) {$u_{\i}$};
        }
      \draw (u2) -- (u4) -- (u3) -- (u5) -- (u1) -- (u2);
      \draw (u2) -- (v2) (u5) -- (v5);
    \end{tikzpicture}
    \caption{$(1\mathring{2}43\mathring{5}1)$}
  \end{subfigure}

  \begin{subfigure}[b]{3cm}
    \centering
    \begin{tikzpicture}[object/.style={thin,double,<->}]
      \foreach \i in {1,..., 10} {
          \draw[thin] ({\i * (360 / 10)}:1) -- ({36 + \i * (360 / 10)}:1);
        }
      \foreach \i in {1,..., 5} {
          \draw[normal edge, olive] ({\i * (360 / 5)}:1) -- ({72 + \i * (360 / 5)}:1);
        }
      \foreach \i in {1,..., 5} {
          \draw[normal edge, violet] ({36 + \i * (360 / 5)}:1) -- ({-36 + \i * (360 / 5)}:1);
        }
      \foreach \i in {1,..., 5} {
          \node[filled vertex, fill = violet] (v\i) at ({36 + \i * (360 / 5)}:1) {};
          \node at ({36 + \i * (360 / 5)}:1.25) {$v_{\i}$};
        }
      \foreach \i in {1,..., 5} {
          \node[filled vertex, fill = olive] (u\i) at ({216 + \i * (360 / 5)}:1) {};
          \node at ({216 + \i * (360 / 5)}:1.25) {$u_{\i}$};
        }
        \draw (u1) -- (v1) (u2) -- (v2) (u3) -- (v3) (u4) -- (v4) (u5) -- (v5);
    \end{tikzpicture}
    \caption{$(\mathring{1}\mathring{2}\mathring{3}\mathring{4}\mathring{5}\mathring{1})$}
  \end{subfigure}
  \begin{subfigure}[b]{3cm}
    \centering
    \begin{tikzpicture}[object/.style={thin,double,<->}]
      \foreach \i in {1,..., 10} {
          \draw[thin] ({\i * (360 / 10)}:1) -- ({36 + \i * (360 / 10)}:1);
        }
      \foreach \i in {1,..., 5} {
          \draw[normal edge, olive] ({\i * (360 / 5)}:1) -- ({72 + \i * (360 / 5)}:1);
        }
      \foreach \i in {1,..., 5} {
          \draw[normal edge, violet] ({36 + \i * (360 / 5)}:1) -- ({-36 + \i * (360 / 5)}:1);
        }
      \foreach \i in {1,..., 5} {
          \node[filled vertex, fill = violet] (v\i) at ({36 + \i * (360 / 5)}:1) {};
          \node at ({36 + \i * (360 / 5)}:1.25) {$v_{\i}$};
        }
      \foreach \i in {1,..., 5} {
          \node[filled vertex, fill = olive] (u\i) at ({216 + \i * (360 / 5)}:1) {};
          \node at ({216 + \i * (360 / 5)}:1.25) {$u_{\i}$};
        }
        \draw (u1) -- (v1) (u3) -- (v3) (u4) -- (v4) (u5) -- (v5);
    \end{tikzpicture}
    \caption{$(\mathring{1}2\mathring{3}\mathring{4}\mathring{5}\mathring{1})$}
  \end{subfigure}
  \begin{subfigure}[b]{3cm}
    \centering
    \begin{tikzpicture}[object/.style={thin,double,<->}]
      \foreach \i in {1,..., 10} {
          \draw[thin] ({\i * (360 / 10)}:1) -- ({36 + \i * (360 / 10)}:1);
        }
      \foreach \i in {1,..., 5} {
          \draw[normal edge, olive] ({\i * (360 / 5)}:1) -- ({72 + \i * (360 / 5)}:1);
        }
      \foreach \i in {1,..., 5} {
          \draw[normal edge, violet] ({36 + \i * (360 / 5)}:1) -- ({-36 + \i * (360 / 5)}:1);
        }
      \foreach \i in {1,..., 5} {
          \node[filled vertex, fill = violet] (v\i) at ({36 + \i * (360 / 5)}:1) {};
          \node at ({36 + \i * (360 / 5)}:1.25) {$v_{\i}$};
        }
      \foreach \i in {1,..., 5} {
          \node[filled vertex, fill = olive] (u\i) at ({216 + \i * (360 / 5)}:1) {};
          \node at ({216 + \i * (360 / 5)}:1.25) {$u_{\i}$};
        }
        \draw (u1) -- (v1) (u2) -- (v2) (u3) -- (v3);
    \end{tikzpicture}
    \caption{$(\mathring{1}\mathring{2}\mathring{3}45\mathring{1})$}
  \end{subfigure}
  \begin{subfigure}[b]{3cm}
    \centering
    \begin{tikzpicture}[object/.style={thin,double,<->}]
      \foreach \i in {1,..., 10} {
          \draw[thin] ({\i * (360 / 10)}:1) -- ({36 + \i * (360 / 10)}:1);
        }
      \foreach \i in {1,..., 5} {
          \draw[normal edge, olive] ({\i * (360 / 5)}:1) -- ({72 + \i * (360 / 5)}:1);
        }
      \foreach \i in {1,..., 5} {
          \draw[normal edge, violet] ({36 + \i * (360 / 5)}:1) -- ({-36 + \i * (360 / 5)}:1);
        }
      \foreach \i in {1,..., 5} {
          \node[filled vertex, fill = violet] (v\i) at ({36 + \i * (360 / 5)}:1) {};
          \node at ({36 + \i * (360 / 5)}:1.25) {$v_{\i}$};
        }
      \foreach \i in {1,..., 5} {
          \node[filled vertex, fill = olive] (u\i) at ({216 + \i * (360 / 5)}:1) {};
          \node at ({216 + \i * (360 / 5)}:1.25) {$u_{\i}$};
        }
      \draw (u2) -- (v2) (u4) -- (v4) (u5) -- (v5);
    \end{tikzpicture}
    \caption{$(1\mathring{2}3\mathring{4}\mathring{5}1)$}
  \end{subfigure}
  \begin{subfigure}[b]{3cm}
    \centering
    \begin{tikzpicture}[object/.style={thin,double,<->}]
      \foreach \i in {1,..., 10} {
          \draw[thin] ({\i * (360 / 10) - 18}:1) -- ({\i * (360 / 10) + 18}:1);
        }
      \foreach \i in {1,..., 5} {
          \draw[normal edge, violet] ({\i * (360 / 5) - 54}:1) -- ({\i * (360 / 5) + 18}:1);
        }
      \foreach \i in {1,..., 5} {
          \node[filled vertex, fill = violet] (v\i) at ({\i * (360 / 5) + 18}:1) {};
          \node at ({\i * (360 / 5) + 18}:1.25) {$v_{\i}$};
        }
      \foreach \i in {1,..., 5} {
          \node[filled vertex, fill = olive] (u\i) at ({\i * (360 / 5) - 162}:1) {};
          \node at ({\i * (360 / 5) - 162}:1.25) {$u_{\i}$};
        }
      \draw (u2) -- (u4) -- (u3) -- (u5) -- (u1) -- (u2);
      \draw (u1) -- (v1) (u2) -- (v2) (u5) -- (v5);
    \end{tikzpicture}
    \caption{$(\mathring{1}\mathring{2}43\mathring{5}\mathring{1})$}
  \end{subfigure}
  \caption{The $(3,3)$-partitionable graphs (the notation will be introduced in Section~\ref{sec:minimal-graphs}).}
  \label{fig: D graphs}
\end{figure}

A particularly nice property of perfect graphs is that they are closed under complementation, known as the weak perfect graph theorem~\cite{lovasz-72-perfect-graphs}, also a trivial corollary of the strong perfect graph theorem.  The key step of proving the weak perfect graph theorem is the Replication Lemma: The class of perfect graphs is closed under (clique) substitution.  As evidenced by $K_4$, t-perfection is closed under neither substitution nor complementation, and this may partially explain the difficulty in characterizing t-perfect graphs.  Benchetrit~\cite{benchetrit-15-thesis} has fully characterized t-perfection with respect to substitution.
The purpose of this paper is to understand t-perfection with respect to complementation.  In particular, we want to know whether there exist minimally t-imperfect graphs whose complements are also minimally t-imperfect.  If we check Figure~\ref{fig:minimally t-imperfect} carefully, we may find that with few exceptions, (${W_3}$, ${W_5}$, ${W_7}$, $\overline{C_7}$, and $\overline{C_{10}^2}$), the complements of all the others contain a $K_4$, hence not minimally t-imperfect graphs.
On the other hand, it is quite obvious that the graphs in the second row of Figure~\ref{fig: D graphs} are precisely the complements of those in the first.
Our main result is that the ten $(3,3)$-partitionable graphs are all the minimally t-imperfect graphs whose complements are also minimally t-imperfect.

\begin{theorem}
  \label{thm:main}
  Let $G$ be a minimally t-imperfect graph.  The complement of ${G}$ is minimally t-imperfect if and only if $G$ is a $(3,3)$-partitionable graph.
\end{theorem}

By the Ramsey theorem, a graph on 18 or more vertices contains a $K_4$ or its complement.  Thus, it suffices to consider graphs of no more than 17 vertices.\footnote{However, there are already 12005168 non-isomorphic graphs of 10 vertices (\url{http://oeis.org/A000088}).}    We start from an easy observation that such a graph contains a $5$-hole.  Fixing a $5$-hole $C$, we study the connection between $C$ and other vertices, and it turns out that the graph cannot have more than ten vertices.   Moreover, all such graphs follow a few simple patterns, and a careful inspection of these patterns leads to the main result.
As a byproduct, we characterize all self-complementary t-perfect graphs that are not perfect.

\begin{theorem}
  \label{thm:self-complement}
  Let $G$ be a self-complementary graph that is not perfect.  Then $G$ is t-perfect if and only if
  $G$ is one of the five graphs in Figure~\ref{fig:self-complement}.
\end{theorem}

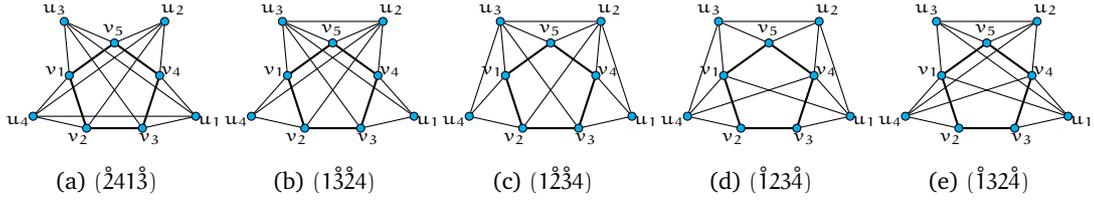
\begin{figure}[h]
  \centering
  \scriptsize
  \begin{subfigure}[b]{2.8cm}
    \centering
    \begin{tikzpicture}[scale=0.5]
      \foreach \i in {1,..., 5} {
          \draw[normal edge] ({\i * (360 / 5) - 54}:1.25) -- ({\i * (360 / 5) + 18}:1.25);
        }
      \foreach \i in {1,..., 4} {
          \draw ({\i * (360 / 5) - 126}:1.25) -- ({\i * (360 / 5) - 90}:2.25) -- ({\i * (360 / 5) - 54}:1.25);
          \node[filled vertex] (u\i) at ({\i * (360 / 5) - 90}:2.25) {};
          \node at ({\i * (360 / 5) - 90}:2.65) {$u_{\i}$};
        }
      \foreach \i in {1,..., 5} {
          \node[filled vertex] (v\i) at ({\i * (360 / 5) - 54}:1.25) {};
          \node at ({\i * (360 / 5) + 90}:1.6) {$v_{\i}$};
        }
      \draw (u2) -- (u4) -- (u1) -- (u3);
      \draw (u2) -- (v4) (u3) -- (v5);
    \end{tikzpicture}
    \caption{$(\mathring{2}41\mathring{3})$}
  \end{subfigure}
  \begin{subfigure}[b]{2.8cm}
    \centering
    \begin{tikzpicture}[scale=0.5]
      \foreach \i in {1,..., 5} {
          \draw[normal edge] ({\i * (360 / 5) - 54}:1.25) -- ({\i * (360 / 5) + 18}:1.25);
        }
      \foreach \i in {1,..., 4} {
          \draw ({\i * (360 / 5) - 126}:1.25) -- ({\i * (360 / 5) - 90}:2.25) -- ({\i * (360 / 5) - 54}:1.25);
          \node[filled vertex] (u\i) at ({\i * (360 / 5) - 90}:2.25) {};
          \node at ({\i * (360 / 5) - 90}:2.65) {$u_{\i}$};
        }
      \foreach \i in {1,..., 5} {
          \node[filled vertex] (v\i) at ({\i * (360 / 5) - 54}:1.25) {};
          \node at ({\i * (360 / 5) + 90}:1.6) {$v_{\i}$};
        }
      \draw (u1) -- (u3) -- (u2) -- (u4);
      \draw (u2) -- (v4) (u3) -- (v5);
    \end{tikzpicture}
    \caption{$(1\mathring{3}\mathring{2}4)$}
  \end{subfigure}
  \begin{subfigure}[b]{2.8cm}
    \centering
    \begin{tikzpicture}[scale=0.5]
      \foreach \i in {1,..., 5} {
          \draw[normal edge] ({\i * (360 / 5) - 54}:1.25) -- ({\i * (360 / 5) + 18}:1.25);
        }
      \foreach \i in {1,..., 4} {
          \draw ({\i * (360 / 5) - 126}:1.25) -- ({\i * (360 / 5) - 90}:2.25) -- ({\i * (360 / 5) - 54}:1.25);
          \node[filled vertex] (u\i) at ({\i * (360 / 5) - 90}:2.25) {};
          \node at ({\i * (360 / 5) - 90}:2.65) {$u_{\i}$};
        }
      \foreach \i in {1,..., 5} {
          \node[filled vertex] (v\i) at ({\i * (360 / 5) - 54}:1.25) {};
          \node at ({\i * (360 / 5) + 90}:1.6) {$v_{\i}$};
        }
      \draw (u1) -- (u2) -- (u3) -- (u4);
      \draw (u2) -- (v4) (u3) -- (v5);
    \end{tikzpicture}
    \caption{$(1\mathring{2}\mathring{3}4)$}
  \end{subfigure}
  \begin{subfigure}[b]{2.8cm}
    \centering
    \begin{tikzpicture}[scale=0.5]
      \foreach \i in {1,..., 5} {
          \draw[normal edge] ({\i * (360 / 5) - 54}:1.25) -- ({\i * (360 / 5) + 18}:1.25);
        }
      \foreach \i in {1,..., 4} {
          \draw ({\i * (360 / 5) - 126}:1.25) -- ({\i * (360 / 5) - 90}:2.25) -- ({\i * (360 / 5) - 54}:1.25);
          \node[filled vertex] (u\i) at ({\i * (360 / 5) - 90}:2.25) {};
          \node at ({\i * (360 / 5) - 90}:2.65) {$u_{\i}$};
        }
      \foreach \i in {1,..., 5} {
          \node[filled vertex] (v\i) at ({\i * (360 / 5) - 54}:1.25) {};
          \node at ({\i * (360 / 5) + 90}:1.6) {$v_{\i}$};
        }
      \draw (u1) -- (u2) -- (u3) -- (u4);
      \draw (u1) -- (v3) (u4) -- (v1);
    \end{tikzpicture}
    \caption{$(\mathring{1}23\mathring{4})$}
  \end{subfigure}
  \begin{subfigure}[b]{2.8cm}
    \centering
    \begin{tikzpicture}[scale=0.5]
      \foreach \i in {1,..., 5} {
          \draw[normal edge] ({\i * (360 / 5) - 54}:1.25) -- ({\i * (360 / 5) + 18}:1.25);
        }
      \foreach \i in {1,..., 4} {
          \draw ({\i * (360 / 5) - 126}:1.25) -- ({\i * (360 / 5) - 90}:2.25) -- ({\i * (360 / 5) - 54}:1.25);
          \node[filled vertex] (u\i) at ({\i * (360 / 5) - 90}:2.25) {};
          \node at ({\i * (360 / 5) - 90}:2.65) {$u_{\i}$};
        }
      \foreach \i in {1,..., 5} {
          \node[filled vertex] (v\i) at ({\i * (360 / 5) - 54}:1.25) {};
          \node at ({\i * (360 / 5) + 90}:1.6) {$v_{\i}$};
        }
      \draw (u1) -- (u3) -- (u2) -- (u4);
      \draw (u1) -- (v3) (u4) -- (v1);
    \end{tikzpicture}
    \caption{$(\mathring{1}32\mathring{4})$}
  \end{subfigure}
  \caption{Self-complementary t-perfect graphs that contain a $C_{5}$, shown by thick lines (the notation will be introduced in Section~\ref{sec:minimal-graphs}).}
  \label{fig:self-complement}
\end{figure}

All the other self-complementary t-perfect graphs are perfect.
Although there is an infinite number of self-complementary graphs that are perfect, e.g., obtained by the 4-path addition~\cite{kawarabayashi-02-separable-self-complementary-graphs}, almost all of them contain a $K_4$, hence not t-perfect.  Indeed, with the same argument by Ramsey theorem, a self-complementary t-perfect graph has at most 17 vertices.  We believe the number is very small.

\section{Core graphs}
\label{sec:minimal-graphs}

All graphs discussed in this paper are undirected and simple.  The vertex set and edge set of a graph $G$ are denoted by, respectively, $V(G)$ and $E(G)$.
Throughout the paper, we use $n$ to denote $|V(G)|$, the \emph{order} of $G$.
The two ends of an edge are \emph{neighbors} of each other, and the number of neighbors of $v\in V(G)$, denoted by $d(v)$, is its degree.  For a subset $U\subseteq V(G)$, let $G[U]$ denote the subgraph of $G$ induced by $U$, whose vertex set is $U$ and whose edge set comprises all the edges whose both ends are in $U$, and let $G - U = G[V(G)\setminus U]$, which is simplified as $G-v$ if $U$ comprises of a single vertex $v$.
A \emph{clique} is a set of pairwise adjacent vertices, and an \emph{independent set} is a set of vertices that are pairwise nonadjacent.
The complement $\overline{G}$ of a graph $G$ is defined on the same vertex set as $G$ and two distinct vertices of $\overline{G}$ are adjacent if and only if they are not adjacent in $G$.  Note that a clique of $G$ is an independent set of $\overline G$.
A graph is \emph{almost bipartite} if there is a vertex whose deletion leaves the graph bipartite.

For $\ell \geq 1$, we use $P_{\ell}$ and $K_{\ell}$ to denote the path graph and complete graph, respectively, on $\ell$ vertices.
For $\ell \geq 3$, we use $C_{\ell}$ and $W_{\ell}$ to denote, respectively, the $\ell$-cycle and the \textit{$\ell$-wheel}, which is obtained from a $C_\ell$ by adding a new vertex and making it adjacent to all vertices on the cycle; note that $W_3$ is precisely $K_4$.  For $\ell \ge 4$, an induced $\ell$-cycle is also called an \emph{$\ell$-hole}.
An $\ell$-cycle, $\ell$-hole, or $\ell$-wheel is \emph{odd} if $\ell$ is odd.
For $k \geq 1$, the \emph{$k$th power} of $C_{\ell}$, denoted by $C_{\ell}^{k}$, is obtained from $C_{\ell}$ by adding an edge between any two vertices of distance at most $k$.  For $k\ge 2$, the $(2k)$th \emph{even M\"{o}bius ladder} is $\overline{C_{4k}^{2k-2}}$.
For integers $p, q \geq 2$, a graph $G$ is \textit{$(p,q)$-partitionable} if $n=pq+1$ and for every vertex $v \in V(G)$, the set $V(G)\setminus \{v\}$ can be partitioned into $q$ independent sets of order $p$ and can be partitioned into $p$ cliques of order $q$.

The \emph{independent set polytope} of a graph $G$ is the convex hull of the characteristic vectors of all independent sets in $G$.
For a graph $G$, let $P(G)$ denote the polytope defined by
\begin{align*}
  0 \leq x_v & \leq 1                  &  & \text{for every vertex } v \in V(G),
  \\
  x_u + x_v  & \leq 1                  &  & \text{for every edge } u v\in E(G),
  \\
  x(V(C))    & \leq \frac{|V(C)|-1}{2} &  & \text{for every induced odd cycle $C$ in } G.
\end{align*}
It is worth pointing out that the last constraints can be imposed on all odd cycles instead of only induced ones, and there would be more constraints, but they determine the same polytope.
Since the characteristic vector of every independent set of $G$ is in $P(G)$, the {independent set polytope} of a graph $G$ is contained in $P(G)$, while the other direction is not true in general.
A graph $G$ is \emph{t-perfect} if $P(G)$ is precisely the independent set polytope of $G$.
It is not difficult to see that every vector in the independent set polytope of $G$ also satisfies the clique constraints
\[
  \sum_{v\in K} x_v \leq 1\qquad \text{for every clique } K \text{ of } G.
\]
Since the vector $(\frac{1}{3},\frac{1}{3},\frac{1}{3},\frac{1}{3})^{\mathsf{T}}$ is in $P(K_{4})$ but does not satisfy the clique constraint, $K_4$ is not t-perfect.

A graph $G$ is \emph{perfect} if neither $G$ nor $\overline G$ contains an odd hole~\cite{berge-60-perfect-conjecture, chudnovsky-06-strong-perfect-graphs-theorem}.
The independent set polytope of a perfect graph is determined by non-negativity and clique inequalities~\cite{chvatal-75-graph-polytopes, padberg-74-perfect-matrices}.
If a perfect graph $G$ contains no $K_{4}$, then every clique of $G$ has order at most three, and hence any clique constraint in $G$ is one of the three in the definition of $P(G)$.

\begin{proposition}\label{prop:perfect and t-perfect}
  Every $K_{4}$-free perfect graph is t-perfect.
\end{proposition}

It is easy to verify that t-perfection is preserved under vertex deletions: For every $v \in V(G)$, the polytope $P(G - v)$ is the intersection of $P(G)$ with the face $x_{v} = 0$.
Moreover, t-perfection is also preserved under \emph{t-contractions} at a vertex $v$ with $N(v)$ being an independent set---contracting $N(v)\cup \{v\}$ into a single vertex~\cite{gerards-98-all-subgraphs-t-perfect}.
Any graph $H$ that can be obtained from $G$ by a sequence of vertex deletions and t-contractions is a \emph{t-minor} of $G$, and $H$ is a \textit{proper t-minor} of $G$ if $H$ has fewer vertices than $G$.
Therefore, t-perfection is closed under taking t-minors.
A graph is \emph{minimally t-imperfect} if it is t-imperfect but all its proper t-minors are t-perfect, e.g., $K_4$.

We say that a graph $G$ is a \emph{core graph} if neither $G$ nor its complement contains a t-imperfect graph as a proper t-minor.  By definition, any t-minor of a core graph is also a core graph.  Moreover, if $G$ is a core graph, then $G$ is either t-perfect or minimally t-imperfect, and so is $\overline{G}$; it is possible that $G$ is t-perfect while $\overline{G}$ is minimally t-imperfect, e.g., $C_7$ and $\overline{C_7}$.  However, there are t-perfect graphs that are not core graphs, e.g., $C_9$ and $\overline{K_5}$.

\begin{proposition}\label{lem:K4 and I4 free}
  A core graph cannot contain a $K_4$ or its complement as a proper induced subgraph.
\end{proposition}

By Proposition~\ref{prop:perfect and t-perfect}, any $\{K_4, \overline{K_4}\}$-free perfect graph is a core graph. Therefore, we focus on core graphs that are not perfect. Such a graph cannot contain an odd hole longer than seven or its complement as a proper induced subgraph.

\begin{proposition} \label{lem:odd hole constrain}
  Let $G$ be a core graph different from $C_{7}$ and $\overline{C_{7}}$.
  Every odd hole in $G$ is a $C_5$.
  Moreover, if $G$ is t-imperfect, then $G$ contains a $C_5$.
\end{proposition}
\begin{proof}
  For the first assertion, note that $\overline{C_{7}}$ is t-imperfect, so the only core graph that contains $C_7$ as an induced subgraph is $C_7$ itself; and for $k \geq 4$, the hole $C_{2k+1}$ contains a $\overline{K_{4}}$.
  For the second assertion, note that if $G$ does not contain a $C_5$, then $G$ is perfect, hence t-perfect by Propositions~\ref{prop:perfect and t-perfect} and \ref{lem:K4 and I4 free}.
\end{proof}

As we will see, $5$-holes are pivotal in core graphs.  First, every $C_5$ in a core graph different from $\overline{W_{5}}$ is dominating: Every other vertex is adjacent to at least two vertices on it.

\begin{lemma}\label{lem:connection of u to C}
  Let $G$ be a core graph different from $W_{5}$ and its complement. If $G$ contains a $5$-hole $C$, then for every $u \in V(G) \setminus V(C)$, either
  \begin{enumerate}[i)]
    \item $u$ has exactly two neighbors on $C$, and they are consecutive on $C$; or
    \item $u$ has exactly three neighbors on $C$, and they are not consecutive on $C$.
  \end{enumerate}
\end{lemma}
\begin{proof}
  We consider the subgraph $G'$ of $G$ induced by $u$ and the five vertices on $C$.
  If $u$ is adjacent to all vertices on $C$, then $G'$ is a $W_5$.  Since $W_5$ is t-imperfect, $G = G'$, a contradiction.
  If $u$ is adjacent to four vertices or three consecutive vertices on $C$, then $K_{4}$ is a proper t-minor of $G'$, with $t$-contraction at a non-neighbor of $u$ on $C$.
  Noting that the complement of $C$ is a $C_5$, we end with the same contradictions on $\overline G$ if $u$ has zero or one neighbor on $C$, or its two neighbors on $C$ are not consecutive.
\end{proof}

The next proposition further stipulates the relationship between a $5$-hole and other vertices in a core graph.
\begin{proposition} \label{lem:duplicate}
  In a core graph, every pair of consecutive vertices on a $5$-hole has at most one common neighbor.
\end{proposition}
\begin{proof}
  Let $G$ be a core graph, and let $v_{1} v_{2} v_{3} v_{4} v_{5}$ be a $5$-hole in $G$.  Suppose for contradiction that there are two vertices $x, y\in N(v_2)\cap N(v_3)$.  By Lemma~\ref{lem:connection of u to C}, neither of $x$ and $y$ is adjacent to $v_{1}$ or $v_{4}$.  But then dependent on whether they are adjacent, $x$ and $y$ either form a $K_{4}$ with $\{v_{2}, v_{3}\}$, or a $\overline{K_{4}}$ with $\{v_{1}, v_{4}\}$, both contradicting Proposition~\ref{lem:K4 and I4 free}.  The same argument applies to other edges on the $5$-cycle.
\end{proof}

As a consequence of Proposition~\ref{lem:K4 and I4 free} and Ramsey theorem, a core graph has at most 17 vertices.  Propositions~\ref{lem:connection of u to C} and \ref{lem:duplicate} together imply a tighter upper bound on those that are not perfect.

\begin{corollary} \label{cor:ten vertices}
  If a core graph contains a $C_5$, then it has at most ten vertices.
\end{corollary}

Let $G$ be a core graph that contains a $5$-hole, and we use the following notations for its vertices and edges, where the indices are always understood as modulo $5$.  We fix a $5$-hole $C$ and number its vertices as $v_{1},\ldots, v_{5}$ in order, and let $U = V(G)\setminus V(C)$.
According to Lemma~\ref{lem:connection of u to C}, each vertex in $U$ is adjacent to two consecutive vertices on $C$.
If a vertex in $U$ is adjacent to $v_{i}$ and $v_{i+1}$, $i=1,\ldots,5$, then we denote it as $u_{i+3}$; by Lemma~\ref{lem:duplicate}, this is well defined.
The five edges on $C$ are all the edges among $v_{1},\ldots, v_{5}$.  For each $u_i$, the two edges $u_i v_{i+2}$ and $u_i v_{i+3}$ must exist in $G$.
Apart from these $2 |U| + 5$ edges, by Lemma~\ref{lem:connection of u to C}, the other possible edges are among $U$ or $u_i v_i$, $i=1,\ldots,5$; they are called \emph{potential edges}.
Shown in Figures~\ref{fig:possible configurations of G}(a, b) are two \emph{pattern graphs}, from which
we can obtain different particular graphs, with different materializations of potential edges.
We use $(1324)$ to denote the graph of pattern Figure~\ref{fig:possible configurations of G}(b) in which $U$ induces a path, with edges $u_1 u_3$, $u_2 u_3$, and $u_2 u_4$.
In case that $G[U]$ is not connected, we use $\|$ to separate its components, e.g., $(14\|23)$ in Figure~\ref{fig:possible configurations of G}(d).
Moreover, we cap an index $i$ with $\circ$ to denote the presence of the edge $u_{i} v_{i}$, e.g., $(1\mathring{3}\mathring{2}4)$ in Figure~\ref{fig:possible configurations of G}(c).

\begin{figure}[h]
  \centering
  \begin{subfigure}[b]{4cm}
    \centering
    \begin{tikzpicture} [scale=0.6]
      \foreach \i in {1,..., 5} {
          \draw[normal edge] ({\i * (360 / 5) - 54}:1.25) -- ({\i * (360 / 5) + 18}:1.25);
        }
      \foreach \i in {1,..., 3} {
          \draw[normal edge] ({\i * (360 / 5) - 126}:1.25) -- ({\i * (360 / 5) - 90}:2.25) -- ({\i * (360 / 5) - 54}:1.25);
          \node[filled vertex] (u\i) at ({\i * (360 / 5) - 90}:2.25) {};
          \node at ({\i * (360 / 5) - 90}:2.75) {$u_{\i}$};
        }
      \foreach \i in {1,..., 5} {
          \node[filled vertex] (v\i) at ({\i * (360 / 5) + 90}:1.25) {};
          \node at ({\i * (360 / 5) + 90}:1.75) {$v_{\i}$};
        }

      \draw [potential edge] (u1) edge [bend right] (u3)  (u1) -- (u2) (u2) -- (u3);
      \draw [potential edge] (u1) -- (v1) (u2) -- (v2) (u3) -- (v3);
    \end{tikzpicture}
    \caption{A pattern on 8 vertices}
  \end{subfigure}
  \begin{subfigure}[b]{4cm}
    \centering
    \begin{tikzpicture} [scale = 0.55]
      \foreach \i in {1,..., 5} {
          \draw[normal edge] ({\i * (360 / 5) - 54}:1.25) -- ({\i * (360 / 5) + 18}:1.25);
        }
      \foreach \i in {1,..., 4} {
          \draw[normal edge] ({\i * (360 / 5) - 126}:1.25) -- ({\i * (360 / 5) - 90}:2.25) -- ({\i * (360 / 5) - 54}:1.25);
          \node[filled vertex] (u\i) at ({\i * (360 / 5) - 90}:2.25) {};
          \node at ({\i * (360 / 5) - 90}:2.75) {$u_{\i}$};
        }
      \foreach \i in {1,..., 5} {
          \node[filled vertex] (v\i) at ({\i * (360 / 5) + 90}:1.25) {};
          \node at ({\i * (360 / 5) + 90}:1.75) {$v_{\i}$};
        }
      \draw [potential edge] (u1) -- (u2) -- (u3) -- (u4);
      \draw [potential edge] (u1) -- (u4) (u1) -- (u3) (u2) -- (u4);
      \draw [potential edge] (u1) -- (v1) (u2) -- (v2) (u3) -- (v3) (u4) -- (v4);
    \end{tikzpicture}
    \caption{A pattern on 9 vertices}
  \end{subfigure}
  \begin{subfigure}[b]{.24\textwidth}
    \centering
    \begin{tikzpicture}[scale=0.5]
      \foreach \i in {1,..., 5} {
          \draw ({\i * (360 / 5) - 54}:1.25) -- ({\i * (360 / 5) + 18}:1.25);
        }
      \foreach \i in {1,..., 4} {
          \draw ({\i * (360 / 5) - 126}:1.25) -- ({\i * (360 / 5) - 90}:2.25) -- ({\i * (360 / 5) - 54}:1.25);
          \node[filled vertex] (u\i) at ({\i * (360 / 5) - 90}:2.25) {};
          \node at ({\i * (360 / 5) - 90}:2.75) {$u_{\i}$};
        }
      \foreach \i in {1,..., 5} {
          \node[filled vertex] (v\i) at ({\i * (360 / 5) + 90}:1.25) {};
          \node at ({\i * (360 / 5) + 90}:1.75) {$v_{\i}$};
        }
      \draw (u1) -- (u3) (u2) -- (u4);
      \draw (u2) -- (u3);
      \draw (u2) -- (v2) (u3) -- (v3);
    \end{tikzpicture}
    \caption{$(1\mathring{3}\mathring{2}4)$}
  \end{subfigure}
  \begin{subfigure}[b]{.24\textwidth}
    \centering
    \begin{tikzpicture}[scale=0.5]
      \foreach \i in {1,..., 5} {
          \draw ({\i * (360 / 5) - 54}:1.25) -- ({\i * (360 / 5) + 18}:1.25);
        }
      \foreach \i in {1,..., 4} {
          \draw ({\i * (360 / 5) - 126}:1.25) -- ({\i * (360 / 5) - 90}:2.25) -- ({\i * (360 / 5) - 54}:1.25);
          \node[filled vertex] (u\i) at ({\i * (360 / 5) - 90}:2.25) {};
          \node at ({\i * (360 / 5) - 90}:2.75) {$u_{\i}$};
        }
      \foreach \i in {1,..., 5} {
          \node[filled vertex] (v\i) at ({\i * (360 / 5) + 90}:1.25) {};
          \node at ({\i * (360 / 5) + 90}:1.75) {$v_{\i}$};
        }
      \draw (u1) -- (u4);
      \draw (u2) -- (u3);
    \end{tikzpicture}
    \caption{$(14\|23)$}
  \end{subfigure}
  \caption{Two patterns (a, b) and two particular graphs (c, d) of the second pattern.  In the patterns, potential edges are depicted as thin green lines, while normal ones as thick black lines; no other edges can exist.}
  \label{fig:possible configurations of G}
\end{figure}
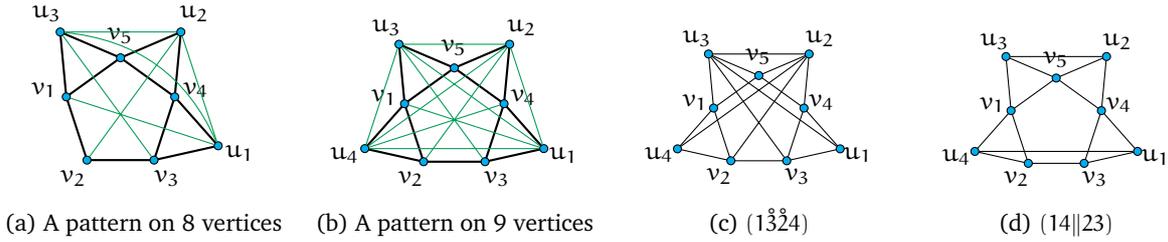

The proof of the following proposition is given in the Appendix.

\begin{proposition}\label{lem:t-perfect-graphs}
  The following graphs are t-perfect: $(12)$, $(1\|\mathring{2})$, $(1\mathring{2})$, $(\mathring{1}\|\mathring{2})$, $(\mathring{1}\mathring{2})$, $(1\|23)$, $(\mathring{3}1\mathring{2})$, $(\mathring{1}\mathring{3}\|\mathring{2})$, $(1\|\mathring{2}4)$, $(14\mathring{2})$, $(1\|\mathring{2}\mathring{4})$, $(1\mathring{4}\mathring{2})$, $(\mathring{1}\|\mathring{2}\|\mathring{4})$, $(\mathring{1}\mathring{2}4)$, $(\mathring{1}\mathring{2}\mathring{4})$, $(1\mathring{3}\mathring{4}\mathring{2})$, $(1\mathring{3}\mathring{4}2)$, $(\mathring{1}\mathring{3}\mathring{4}2)$, $(\mathring{1}\mathring{2}\mathring{4}3)$, $(2\mathring{3}14)$, $(2\mathring{3}\mathring{1}4)$, $(23\mathring{1}4)$, $(1\mathring{4}32)$, $(\mathring{1}\mathring{2}\mathring{4}\mathring{3}\mathring{1})$, $(1\mathring{2}\mathring{4}3)$, $(1\mathring{3}\mathring{2}41)$, $(1\mathring{3}\mathring{2}4)$, $(14\|23)$, $(1\mathring{4}\mathring{3}2)$, $(1\mathring{3}\|\mathring{2}4)$, and $(\mathring{2}41\mathring{3})$.
\end{proposition}

As easy consequences of Lemma~\ref{lem:connection of u to C}, we have the following observations on core graphs.  Here $i = 1, \ldots, 5$.

\begin{enumerate}[Obs.1)]
  \item\label{ob:1} If both $u_{i} v_{i}$ and $u_{i+1} u_{i+2}$ are in $E(G)$, then at least one of $u_{i} u_{i+1}$ and $u_{i} u_{i+2}$ is in $E(G)$; otherwise, $v_{i+4}$ has four neighbors on the $5$-cycle $u_{i} v_{i} u_{i+2} u_{i+1} v_{i+3}$.
        By symmetry, if both $u_{i} v_{i}$ and $u_{i-1} u_{i-2}$ are in $E(G)$, then at least one of $u_{i} u_{i-2}$ and $u_{i} u_{i-1}$ is in $E(G)$.
  \item\label{ob:3} If both $u_{i} u_{i+1}$ and $u_{i} u_{i+3}$ are in $E(G)$, then at least one of $u_{i} v_{i}$ and $u_{i+1} u_{i+3}$ is in $E(G)$; otherwise, $v_{i+3}$ has three consecutive neighbors on the $5$-cycle $u_{i} u_{i+1} v_{i+4} v_{i} u_{i+3}$.
        By symmetry, if both $u_{i} u_{i-1}$ and $u_{i} u_{i-3}$ are in $E(G)$, then at least one of $u_{i} v_{i}$ and $u_{i-1} u_{i-3}$ is in $E(G)$.
  \item\label{ob:5} Suppose, all of $u_{i-2} u_{i-1}$, $u_{i-1} u_{i+1}$, and $u_{i+1} u_{i+2}$ are in $E(G)$.  If $u_{i-1} v_{i-1}$ or $u_{i+1} v_{i+1}$ is in $E(G)$, then at least one of $u_{i-1} u_{i+2}$, $u_{i-2} u_{i+1}$, and $u_{i-2} u_{i+2}$ is in $E(G)$; otherwise, $v_{i-1}$ or $v_{i+1}$ has four neighbors on the $5$-cycle $u_{i-2} u_{i-1} u_{i+1} u_{i+2} v_{i}$.
  \item\label{ob:7}
    If  $u_{i-1} u_{i+1}\in E(G)$ and $u_{i-1} v_{i-1}, u_{i+1} v_{i+1}\not\in E(G)$, then
  $u_{i+1} u_{i+2}, u_{i-1} u_{i-2}\not\in E(G)$, and
 $u_{i} u_{i-1}, u_{i} u_{i+1}\in E(G)$; otherwise, the neighborhood of $u_{i-2}$, $u_{i+2}$, or, respectively, $u_{i}$ on the $5$-cycle $u_{i-1} u_{i+1} v_{i-1} v_{i} v_{i+1}$ does not satisfy Lemma~\ref{lem:connection of u to C}.

  \item\label{ob:10} If $u_{i} u_{i+1}\not\in E(G)$ and at least one of $u_{i}$ and $u_{i+1}$ is adjacent to $u_{i+3}$, then at most one of $u_{i} v_{i}$ and $u_{i+1} v_{i+1}$ can be in $E(G)$; otherwise, $u_{i+3}$ has three consecutive neighbors on the $5$-cycle $u_{i} v_{i} v_{i+1} u_{i+1} v_{i+3}$.

  \item\label{ob:12} If $u_{i+1} v_{i+1}$ is in $E(G)$ and none of $u_{i+1} u_{i+2}$, $u_{i+2} u_{i-2}$, and $u_{i-1} u_{i-2}$ is in $E(G)$, then $u_{i+1} u_{i-2}$, $u_{i+2} u_{i-1}$, and $u_{i+1} u_{i-1}$ cannot be all present in $G$; otherwise, $v_{i+1}$ has four neighbors on the $5$-cycle $u_{i-1} u_{i+2} v_{i} u_{i-2} u_{i+1}$.
        By symmetry, if $u_{i-1} v_{i-1}$ is in $E(G)$ and none of $u_{i+1} u_{i+2}$, $u_{i+2} u_{i-2}$, and $u_{i-1} u_{i-2}$ is in $E(G)$, then $u_{i+1} u_{i-2}$, $u_{i+2} u_{i-1}$, and $u_{i+1} u_{i-1}$ cannot be all present in $G$.
\end{enumerate}

All graphs of pattern Figure~\ref{fig:possible configurations of G}(a) are summarized in Table \ref{table:1} and characterized in Lemma~\ref{lem:t-perfect-order-8}.

\begin{table}[h]
  \centering
  \caption{Graphs of pattern Figure~\ref{fig:possible configurations of G}(a).  The columns are for combinations of edges among $U$; the cases with only $u_2 u_3$ and only $\{u_1 u_3, u_2 u_3\}$ are omitted because they are symmetric to respectively, $u_1 u_2$ and $\{u_1 u_2, u_1 u_3\}$.  The rows are possible combinations of edges between $U$ and $C$.  The invocation of an observation means that this configuration violates this observation.}
  \label{table:1}
  \begin{tabular}{ l | c c c c c}
    \toprule
                           & all                                                                 & $\{u_1 u_2, u_1 u_3\}$        & $\{u_1 u_2, u_2 u_3\}$        & $\{u_1 u_2\}$         & $\{u_1 u_3\}$
    \\
    \midrule
    all                    & $d(u_1) = 5$                                                        & $d(u_1) = 5$                  & $d(u_2) = 5$                  & Obs.\ref{ob:1} ($i=3$) & $(\mathring{1}\mathring{3} \|\mathring{2})$
    \\ [1ex]
    $\{u_1 v_1, u_2 v_2\}$ & $d(u_1) = 5$                                                        & $d(u_1) = 5$                  & $d(u_2) = 5$                  & $d(u_3) = 2$          & $(\mathring{1}3\|\mathring{2})$
    \\ [1ex]
    $\{u_1 v_1, u_3 v_3\}$ & $d(u_1) = 5$                                                        & $d(u_1) = 5$                  & $(\mathring{1}2\mathring{3})$ & Obs.\ref{ob:1} ($i=3$) & $d(u_2) = 2$
    \\ [1ex]
    $\{u_2 v_2, u_3 v_3\}$ & $d(u_2) = 5$                                                        & $(\mathring{3}1\mathring{2})$ & $d(u_2) = 5$                  & Obs.\ref{ob:1} ($i=3$) & $\cong (\mathring{1}3\|\mathring{2})$
    \\ [1ex]
    $\{u_1 v_1\}$          & $d(u_1) = 5$                                                        & $d(u_1) = 5$                  & $(\mathring{1}23)$            & $d(u_3) = 2$          & $d(u_2) = 2$
    \\ [1ex]
    $\{u_2 v_2\}$          & $d(u_2) = 5$                                                        & Obs.\ref{ob:7} ($i = 2$)       & $d(u_2) = 5$                  & $d(u_3) = 2$          & Obs.\ref{ob:7} ($i = 2$)
    \\ [1ex]
    $\{u_3 v_3\}$          & $d(u_3) = 5$                                                        & $(\mathring{3}12)$            & $\cong (\mathring{1}23)$      & Obs.\ref{ob:1} ($i=3$) & $d(u_2) = 2$
    \\ [1ex]
    none                   & $\overline{G} \cong (\mathring{1}  \|\mathring{2}  \|\mathring{4})$ & Obs.\ref{ob:7} ($i = 2$)       & $(123)$                       & $d(u_3) = 2$          & $d(u_2) = 2$
    \\
    \bottomrule
  \end{tabular}
\end{table}

\begin{lemma}\label{lem:t-perfect-order-8}
  Let $G$ be a core graph of order eight.  At least one of $G$ and $\overline{G}$
  \begin{enumerate}[i)]
    \item is t-perfect; or
    \item has a degree-2 vertex in $U$.
  \end{enumerate}
\end{lemma}

\begin{proof}
  Note that if the degree of a vertex is five in $G$, then its degree in $\overline G$ is two.
  According to Table~\ref{table:1}, it suffices to show that graphs $(\mathring{1}\mathring{3}\|\mathring{2})$, $(\mathring{1}3\|\mathring{2})$, $(\mathring{1}2\mathring{3})$, $(\mathring{3}1\mathring{2})$, $(\mathring{1}23)$, $(\mathring{3}12)$, $(\mathring{1}\|\mathring{2}\|\mathring{4})$, and $(123)$ are t-perfect.
  We have seen in Proposition~\ref{lem:t-perfect-graphs} that $(\mathring{1}\|\mathring{2}\|\mathring{4})$, $(\mathring{3}1\mathring{2})$, and $(\mathring{1}\mathring{3}\|\mathring{2})$ are t-perfect.  The graph $(\mathring{1}3\|\mathring{2})$ is t-perfect because $(\mathring{1}3\|\mathring{2})$ is isomorphic to $(\mathring{2}41\mathring{3}) - u_{1}$, and $(\mathring{2}41\mathring{3})$ is t-perfect. On the other hand, $(\mathring{3}12)$, $(\mathring{1}2\mathring{3})$, $(\mathring{1}23)$, and $(123)$ are isomorphic to, respectively, $(1\mathring{2}43\mathring{5}1)-\{u_{1}, u_{2}\}$, $(1\mathring{2}43\mathring{5}1)-\{u_{3}, u_{4}\}$, $(1\mathring{2}3451)-\{u_{1}, u_{5}\}$, and $(123451)-\{u_{1}, u_{5}\}$, all t-perfect.
\end{proof}

\section{Degree-bounded core graphs of order nine}
\label{sec:graph-with-pattern-g}

According to Propositions~\ref{lem:connection of u to C} and \ref{lem:duplicate}, every core graph of order nine is of the pattern in Figure~\ref{fig:possible configurations of G}(b).  Throughout this section, let $G$ denote a core graph of order nine where the degree of every vertex is between three and five.  (The reason of imposing degree constraints will become clear shortly.)  We consider whether edges $u_{i} u_{i+1}$, $i = 1,2,3$ are present in $G$.

\begin{proposition}\label{prop:contains-all-outer-edges}
  Let $G$ be a degree-bounded core graph on nine vertices.  If for all $i = 1,2,3$, the edge $u_i u_{i+1}$ is in $E(G)$, then $G$ is an induced subgraph of a $(3, 3)$-partitionable graph.
\end{proposition}
\begin{proof}
  We argue first that none of $u_{1} u_{4}$, $u_{1} u_{3}$, and $u_{2} u_{4}$ can be present in $G$; i.e., $u_1 u_2 u_3 u_4$ is an induced path in $G$.  Suppose that $u_{1} u_{4} \in E(G)$, then by Obs.\ref{ob:7} (with $i = 5$), at least one of $u_{4} v_{4}$ and $u_{1} v_{1}$ is in $E(G)$.  We may assume that $u_{4} v_{4} \in E(G)$, and the other case is symmetric.  Since $\{u_{1}, u_{4}, u_{2}, v_{4}\}$ is not a clique, $u_{2} u_{4} \notin E(G)$.  By Obs.\ref{ob:3} (with $i=1$), $u_{1} v_{1} \in E(G)$, and then since $\{u_{1}, u_{3}, u_{4}, v_{1}\}$ is not a clique, $u_{1} u_{3}$ cannot be present.
  But then $G- \{v_{2}, v_{3}\}$ is isomorphic to $\overline{C_7}$, a contradiction.  Thus, $u_{1} u_{4} \notin E(G)$.
  By Obs.\ref{ob:3} (with $i=2$), (noting $u_{1} u_{2}\in E(G)$,) the presence of $u_{2} u_{4}$ would imply the presence of $u_{2} v_{2}$, but then $d(u_{2}) = 6$.
  Thus,
  $u_{2} u_{4} \notin E(G)$, and by a symmetric argument, $u_{1} u_{3} \notin E(G)$.

  Now that none of $u_{1} u_{4}$, $u_{1} u_{3}$, and $u_{2} u_{4}$ is present, we consider all possible combinations of edges $\{u_{i} v_{i}\mid i = 1,\dots,4\}\cap E(G)$.  If none of them is in $E(G)$, then $G$ is isomorphic to $(123451) - u_{1}$.  If all of them are in $E(G)$, then $G$ is isomorphic to $(\mathring{1}\mathring{2}\mathring{3}\mathring{4}\mathring{5}\mathring{1}) - u_{1}$.  If only one $u_i v_i$ is in $E(G)$, then $G$ is isomorphic to $(\mathring{1}2\mathring{3}45\mathring{1}) - u_{3}$ or $(123\mathring{4}\mathring{5}1) - u_{4}$.  If only one $u_i v_i$ is absent, then $G$ is isomorphic to $(\mathring{1}\mathring{2}\mathring{3}45\mathring{1}) - u_{4}$ or $(1\mathring{2}3\mathring{4}\mathring{5}1) - u_{1}$.  Otherwise, exact two of edges $u_i v_i$ are in $E(G)$, then $G$ is isomorphic to one of $(\mathring{1}2\mathring{3}45\mathring{1}) - u_{4}$, $(\mathring{1}\mathring{2}\mathring{3}45\mathring{1}) - u_{1}$, $(\mathring{1}2\mathring{3}45\mathring{1}) - u_{2}$, and $(123\mathring{4}\mathring{5}1) - u_{2}$.
\end{proof}

In the rest, for at least one of $i = 1, 2, 3$, the edge $u_{i}u_{i+1}$ is absent from $G$.
In the second case, we assume that both $u_{1}u_{2}$ and $u_{2}u_{3}$ are absent from $G$; see Figure~\ref{fig:two-missing-nine-vertices}(b).

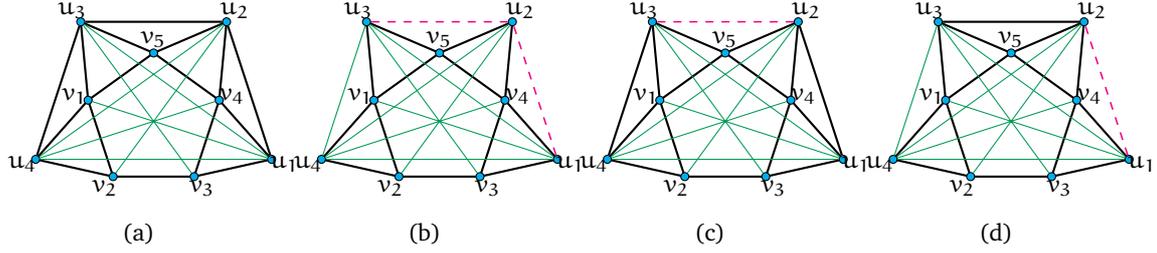
\begin{figure} [h]
  \centering
  \begin{subfigure}[b]{.23\linewidth}
    \centering
    \begin{tikzpicture} [scale = 0.725]
      \foreach \i in {1,..., 5} {
          \draw[normal edge] ({\i * (360 / 5) - 54}:1.25) -- ({\i * (360 / 5) + 18}:1.25);
        }
      \foreach \i in {1,..., 4} {
          \draw[normal edge] ({\i * (360 / 5) - 126}:1.25) -- ({\i * (360 / 5) - 90}:2.25) -- ({\i * (360 / 5) - 54}:1.25);
          \node[filled vertex] (u\i) at ({\i * (360 / 5) - 90}:2.25) {};
          \node at ({\i * (360 / 5) - 90}:2.5) {$u_{\i}$};
        }
      \foreach \i in {1,..., 5} {
          \node[filled vertex] (v\i) at ({\i * (360 / 5) + 90}:1.25) {};
          \node at ({\i * (360 / 5) + 90}:1.5) {$v_{\i}$};
        }
      \draw [normal edge] (u1) -- (u2) -- (u3) -- (u4);
      \draw [potential edge] (u1) -- (u4) (u1) -- (u3) (u2) -- (u4);
      \draw [potential edge] (u1) -- (v1) (u2) -- (v2) (u3) -- (v3) (u4) -- (v4);
    \end{tikzpicture}
    \caption{}
  \end{subfigure}
  \begin{subfigure}[b]{.23\linewidth}
    \centering
    \begin{tikzpicture} [scale = 0.725]
      \foreach \i in {1,..., 5} {
          \draw[normal edge] ({\i * (360 / 5) - 54}:1.25) -- ({\i * (360 / 5) + 18}:1.25);
        }
      \foreach \i in {1,..., 4} {
          \draw[normal edge] ({\i * (360 / 5) - 126}:1.25) -- ({\i * (360 / 5) - 90}:2.25) -- ({\i * (360 / 5) - 54}:1.25);
          \node[filled vertex] (u\i) at ({\i * (360 / 5) - 90}:2.25) {};
          \node at ({\i * (360 / 5) - 90}:2.5) {$u_{\i}$};
        }
      \foreach \i in {1,..., 5} {
          \node[filled vertex] (v\i) at ({\i * (360 / 5) + 90}:1.25) {};
          \node at ({\i * (360 / 5) + 90}:1.5) {$v_{\i}$};
        }
      \draw [forbidden edge] (u1) -- (u2) (u2)-- (u3);
      \draw [potential edge] (u1) -- (u3)-- (u4);
      \draw [potential edge] (u1) -- (u4) (u2) -- (u4);
      \draw [potential edge] (u1) -- (v1) (u2) -- (v2) (u3) -- (v3) (u4) -- (v4);
    \end{tikzpicture}
    \caption{}
  \end{subfigure}
  \begin{subfigure}[b]{.23\linewidth}
    \centering
    \begin{tikzpicture}[scale = 0.725]
      \foreach \i in {1,..., 5} {
          \draw[normal edge] ({\i * (360 / 5) - 54}:1.25) -- ({\i * (360 / 5) + 18}:1.25);
        }
      \foreach \i in {1,..., 4} {
          \draw[normal edge] ({\i * (360 / 5) - 126}:1.25) -- ({\i * (360 / 5) - 90}:2.25) -- ({\i * (360 / 5) - 54}:1.25);
          \node[filled vertex] (u\i) at ({\i * (360 / 5) - 90}:2.25) {};
          \node at ({\i * (360 / 5) - 90}:2.5) {$u_{\i}$};
        }
      \foreach \i in {1,..., 5} {
          \node[filled vertex] (v\i) at ({\i * (360 / 5) + 90}:1.25) {};
          \node at ({\i * (360 / 5) + 90}:1.5) {$v_{\i}$};
        }
      \draw [normal edge] (u1) -- (u2) (u3)-- (u4);
      \draw [forbidden edge] (u2)-- (u3);
      \draw [potential edge] (u1) -- (u3) (u1) -- (u4) (u2) -- (u4);
      \draw [potential edge] (u1) -- (v1) (u2) -- (v2) (u3) -- (v3) (u4) -- (v4);
    \end{tikzpicture}
    \caption{}
    \label{fig:not-contain-u2u3}
  \end{subfigure}
  \begin{subfigure}[b]{.23\linewidth}
    \centering
    \begin{tikzpicture}[scale = 0.725]
      \foreach \i in {1,..., 5} {
          \draw[normal edge] ({\i * (360 / 5) - 54}:1.25) -- ({\i * (360 / 5) + 18}:1.25);
        }
      \foreach \i in {1,..., 4} {
          \draw[normal edge] ({\i * (360 / 5) - 126}:1.25) -- ({\i * (360 / 5) - 90}:2.25) -- ({\i * (360 / 5) - 54}:1.25);
          \node[filled vertex] (u\i) at ({\i * (360 / 5) - 90}:2.25) {};
          \node at ({\i * (360 / 5) - 90}:2.5) {$u_{\i}$};
        }
      \foreach \i in {1,..., 5} {
          \node[filled vertex] (v\i) at ({\i * (360 / 5) + 90}:1.25) {};
          \node at ({\i * (360 / 5) + 90}:1.5) {$v_{\i}$};
        }
      \draw [forbidden edge] (u1) -- (u2);
      \draw [normal edge] (u2)-- (u3);
      \draw [potential edge] (u1) -- (u3) (u1) -- (u4) (u2) -- (u4) (u3)-- (u4);
      \draw [potential edge] (u1) -- (v1) (u2) -- (v2) (u3) -- (v3) (u4) -- (v4);
    \end{tikzpicture}
    \caption{}
  \end{subfigure}
  \caption{Refined patterns on nine vertices, the potential edges in Figure~\ref{fig:possible configurations of G}(b) but absent here are emphasized by red dashed lines.  (a) all the three edges $u_{1} u_{2}$, $u_{2} u_{3}$, and $u_{3} u_{4}$ are present; (b) both $u_{1}u_{2}$ and $u_{2}u_{3}$ are absent; (c) $u_{2} u_{3}$ is absent but both $u_{1} u_{2}$ and $u_{3} u_{4}$ are present; (d) $u_{1} u_{2}$ is absent but $u_{2} u_{3}$ is present.}
  \label{fig:two-missing-nine-vertices}
\end{figure}

\begin{proposition} \label{prop:missing-consecutive-two-nine-vertices-o}
  Let $G$ be a degree-bounded core graph on nine vertices.  If both $u_{1}u_{2}$ and $u_{2}u_{3}$ are absent from $G$, then $G$ is isomorphic to one of $(1\mathring{3}\mathring{4}\mathring{2})$, $(1\mathring{3}\mathring{4}2)$, $(\mathring{1}\mathring{3}\mathring{4}2)$, $(1\mathring{3}\|\mathring{2}4)$, and $(\mathring{2}41\mathring{3})$.
\end{proposition}
\begin{proof}
  We first argue that the edge $u_{1}u_{3}$ must be present.   Suppose for contradiction that  $u_{1}u_{3}$ is absent.
  Note that $u_{2} v_{2} \in E(G)$, as otherwise $\{u_{1},u_{2},u_{3},v_{2}\}$ forms an independent set.  The edge $u_{3} v_{3}$ cannot be in $E(G)$, as otherwise $u_{1}$ has only one neighbor on the $5$-cycle $u_{3} v_{3} v_{2} u_{2} v_{5}$, contradicting Lemma~\ref{lem:connection of u to C}.  Then $d(u_{3}) > 2$ forces $u_{3} u_{4} \in E(G)$.  By Obs.\ref{ob:1} (with $i = 2$), $u_{2} u_{4} \in E(G)$, and by Obs.\ref{ob:10} (with $i = 1$), $u_{1} v_{1} \notin E(G)$.  
Since $d(u_{1}) > 2$, the edge $u_{1} u_{4}$ must be present.  Now that $u_{3} u_{4} \in E(G)$ and $u_{1} v_{1} \not\in E(G)$, Obs.\ref{ob:7} (with $i = 5$) implies $u_{4} v_{4} \in E(G)$. But then $d(u_{4}) = 6$, contradicting that $G$ is degree-bounded.

  Now that $u_{1}u_{3}\in E(G)$, by Obs.\ref{ob:7} (with $i = 2$), at least one of $u_{1} v_{1}$ and $u_{3} v_{3}$ is present.

  Assume first that $u_{1} v_{1}\in E(G)$.  By Proposition~\ref{lem:K4 and I4 free}, at least one of $u_{3} u_{4}$, $u_{2} u_{4}$, and $u_{3} v_{3}$ is in $E(G)$, as otherwise $\{u_{2},u_{3},u_{4},v_{3}\}$ forms an independent set.
  We argue that $u_{3} u_{4} \in E(G)$.
  Suppose for contradiction that $u_{3} u_{4} \notin E(G)$. If $u_{2} u_{4} \in E(G)$, then by Obs.\ref{ob:10} (with $i = 1$), $u_{2} v_{2} \notin E(G)$; and by Obs.\ref{ob:12} (with $i = 5$), Obs.\ref{ob:7} (with $i = 3$), and Obs.\ref{ob:10} (with $i = 3$), $u_{1} u_{4} \notin E(G)$, $u_{4} v_{4} \in E(G)$, and $u_{3} v_{3} \notin E(G)$.  Then $u_{1} v_{3} v_{2} u_{4} u_{2} v_{5} u_{3}$ is a $7$-cycle in $G$, contradicting Proposition~\ref{lem:odd hole constrain}.  Thus, $u_{2} u_{4} \notin E(G)$, and $u_{3} v_{3} \in E(G)$.  By Obs.\ref{ob:10} (with $i = 3$), $u_{4} v_{4} \notin E(G)$.  Since $d(u_{4}) > 2$, $u_{1} u_{4} \in E(G)$ and by Obs.\ref{ob:10} (with $i = 1$), $u_{2} v_{2} \notin E(G)$.  But then $d(u_{2}) = 2$, a contradiction.
  Now that $u_{3} u_{4} \in E(G)$, the edge $u_{1} u_{4}$ cannot exist, as otherwise $\{u_{1}, u_{3}, u_{4}, v_{1}\}$ forms a clique.  By Obs.\ref{ob:3} ($i=3$), $u_{3} v_{3} \in E(G)$.  If $u_{2} v_{2} \in E(G)$, then by Obs.\ref{ob:1} (with $i = 2$), $u_{2} u_{4} \in E(G)$, which violates Obs.\ref{ob:10} (with $i = 1$).  Since $d(u_{2}) > 2$, the edge $u_{2} u_{4}$ must be present.  Obs.\ref{ob:7} (with $i = 3$), together with the fact that  $u_{2} v_{2} \notin E(G)$, implies $u_{4} v_{4} \in E(G)$.  Thus, $G$ is $(\mathring{1}\mathring{3}\mathring{4}2)$.

  In the rest of the proof, $u_{1} v_{1}\not\in E(G)$ and $u_{3} v_{3} \in E(G)$.  Since $d(u_{2}) > 2$, at least one of $u_{2} v_{2}$ and $u_{2} u_{4}$ needs to be present.  Note that the presence of $u_{2} v_{2}$ implies the presence of $u_{2} u_{4}$, by Lemma~\ref{lem:connection of u to C} applied on vertex $u_4$ and the $5$-cycle $u_{3} v_{3} v_{2} u_{2} v_{5}$.
  If $u_{2} u_{4} \in E(G)$, but $u_{2} v_{2}$ is not, then by Obs.\ref{ob:7} (with $i = 3$), $u_{4} v_{4} \in E(G)$.  The edge $u_{3} u_{4}$ is in $E(G)$, as otherwise $u_{1}$ has three consecutive neighbors on the $5$-cycle $u_{3} v_{3} v_{4} u_{4} v_{1}$, contradicting Lemma~\ref{lem:connection of u to C}.  Since $d(u_{4}) < 6$, the edge $u_{1} u_{4}$ cannot be present.  Then $G$ is $(1\mathring{3}\mathring{4}2)$.
  Now that both $u_{2} u_{4}$ and $u_{2} v_{2}$ are present, the only potential edges that have not been excluded are $u_{4} v_{4}$, $u_{3} u_{4}$, and $u_{1} u_{4}$.
  Note that the presence of $u_{3} u_{4}$ implies the presence of $u_{4} v_{4}$; otherwise, by Obs.\ref{ob:7} (with $i = 5$), $u_{1} u_{4}\not\in E(G)$, but then $v_{3}$ has three consecutive neighbors on the $5$-cycle $u_{1} v_{4} u_{2} u_{4} u_{3}$, contradicting Lemma~\ref{lem:connection of u to C}.

  \begin{itemize}
    \item If none of $u_{4} v_{4}$, $u_{3} u_{4}$, and $u_{1} u_{4}$ is present, then $G$ is $(1\mathring{3}\|\mathring{2}4)$.
    \item If $u_{1} u_{4}$ is in $E(G)$ but $u_{4} v_{4}$ and $u_{3} u_{4}$ are not, then $G$ is $(\mathring{2}41\mathring{3})$.
    \item Otherwise, we must have $u_{4} v_{4}\in E(G)$.  Then $u_{3} u_{4}$ must be present as well, as otherwise $u_{3} v_{3} v_{4} u_{4} v_{1}$ is a $5$-cycle on which $u_1$ has three consecutive neighbors, contradicting Lemma~\ref{lem:connection of u to C}.
          Since $d(u_{4}) < 6$, the edge $u_{1} u_{4}$ cannot be present, and $G$ is $(1\mathring{3}\mathring{4}\mathring{2})$.    \qedhere
  \end{itemize}
\end{proof}

Note that it is symmetric to Proposition~\ref{prop:missing-consecutive-two-nine-vertices-o} if both $u_{2}u_{3}$ and $u_{3}u_{4}$ are absent.  Next we consider the situation that $u_{2} u_{3}$ is absent but both $u_{1} u_{2}$ and $u_{3} u_{4}$ are present; see Figure~\ref{fig:two-missing-nine-vertices}(c).

\begin{proposition} \label{prop:not-contain-u2u3}
  Let $G$ be a degree-bounded core graph on nine vertices.  If both $u_{1} u_{2}$ and $u_{3} u_{4}$ are in $E(G)$ but $u_{2} u_{3}$ is not, then $G$ is isomorphic to one of $(1\mathring{2}\mathring{4}3)$, $(\mathring{1}\mathring{2}\mathring{4}\mathring{3}\mathring{1})$, $(\mathring{1}\mathring{2}\mathring{4}3)$, $(1\mathring{2}43\mathring{5}1) - u_{2}$, and $(\mathring{1}\mathring{2}43\mathring{5}\mathring{1}) - u_{5}$.
\end{proposition}

\begin{proof}
  We start by arguing that $u_{1} u_{4}$ is absent, and at least one of $u_{2} v_{2}$ and $u_{3} v_{3}$ is present.
  Suppose for contradiction that $u_{1} u_{4} \in E(G)$.  By Obs.\ref{ob:7} (with $i=5$), at least one of $u_{4} v_{4}$ and $u_{1} v_{1}$ is in $E(G)$. If $u_{4} v_{4}$ is in $E(G)$ but $u_{1} v_{1}$ is not, then by Obs.\ref{ob:3} (with $i=1$), $u_{2} u_{4} \in E(G)$; then $d(u_{4}) = 6$, a contradiction. A symmetric argument applies if $u_{1} v_{1}$ is in $E(G)$ but $u_{4} v_{4}$ is not.  Now that both $u_{1} v_{1}$ and $u_{4} v_{4}$ are in $E(G)$, neither of $u_{1} u_{3}$ and $u_{2} u_{4}$ can be in $E(G)$, as otherwise $\{u_{3},u_{1},u_{4},v_{1}\}$ or, respectively, $\{u_{1},u_{4},u_{2},v_{4}\}$ forms a clique.  But then $v_{4}$ has four neighbors on the $5$-cycle $u_{1} u_{4} u_{3} v_{5} u_{2}$, contradicting Lemma~\ref{lem:connection of u to C}.  In the rest, $u_{1} u_{4}\not\in E(G)$.
For $u_{2} v_{2}$ and $u_{3} v_{3}$, if both of them are absent, then by Obs.\ref{ob:3} (with $i = 3$), $u_{1} u_{3}$ has to be absent as well (note that $u_{3} u_{4}$ is in $E(G)$ while $u_{3} v_{3}$ and $u_{1} u_{4}$ are absent).  By a symmetric argument, the edge $u_{2} u_{4}$ is also absent.  But then $u_{1} v_{3} v_{2} u_{4} u_{3} v_{5} u_{2}$ is a $7$-cycle in $G$, contradicting Proposition \ref{lem:odd hole constrain}.

  Since $u_{2} v_{2}$ and  $u_{3} v_{3}$ are symmetric, it suffices to consider $u_{2} v_{2}\in E(G)$.  By Obs.\ref{ob:1} (with $i=2$), $u_{2} u_{4} \in E(G)$.
  If none of the remaining undecided potential edges, $u_{1} v_{1}$, $u_{3} v_{3}$, $u_{4} v_{4}$, and $u_{1} u_{3}$, is in $E(G)$, then $G$ is isomorphic to $(1\mathring{2}43\mathring{5}1) - u_{2}$.
  If $u_{1} u_{3} \in E(G)$, then by Obs.\ref{ob:3} (with $i = 3$), $u_{3} v_{3} \in E(G)$.  The edge $u_{1} v_{1}$ is in $E(G)$, as otherwise $u_{4}$ has four neighbors on the $5$-cycle $u_{1} u_{3} v_{1} v_{2} u_{2}$, contradicting Lemma~\ref{lem:connection of u to C}.  A symmetric argument enables us conclude that $u_{4} v_{4} \in E(G)$.  Then $G$ is $(\mathring{1}\mathring{2}\mathring{4}\mathring{3}\mathring{1})$.
  Now that $u_{1} u_{3}\not\in E(G)$, which implies $u_{3} v_{3}$ is not in $E(G)$ either, as otherwise, $u_{1} u_{2}$ is in $E(G)$ but neither of $u_{2} u_{3}$ and $u_{1} u_{3}$ is, contracting Obs.\ref{ob:1} (with $i = 3$).  If $u_{1} v_{1}$ is in $E(G)$ but $u_{4} v_{4}$ is not, then $G$ is isomorphic to $(\mathring{1}\mathring{2}43\mathring{5}\mathring{1}) - u_{5}$; if $u_{4} v_{4}$ is in $E(G)$ but $u_{1} v_{1}$ is not, then $G$ is $(1\mathring{2}\mathring{4}3)$; otherwise, both $u_{1} v_{1}$ and $u_{4} v_{4}$ are in $E(G)$, and $G$ is $(\mathring{1}\mathring{2}\mathring{4}3)$.
\end{proof}

In the last case, $u_{2} u_{3}$ is in $E(G)$, but at least one of $u_{1} u_{2}$ and $u_{3} u_{4}$ is not.  We may assume without loss of generality that $u_{1} u_{2}$ is absent; see Figure~\ref{fig:two-missing-nine-vertices}(d).

\begin{proposition} \label{prop:contain-u2u3}
  Let $G$ be a degree-bounded core graph on nine vertices.
  If $u_{2} u_{3}$ is in $E(G)$ but $u_{1} u_{2}$ is not, then $G$ is isomorphic to one of $(2\mathring{3}14)$, $(2\mathring{3}\mathring{1}4)$, $(23\mathring{1}4)$, $(1\mathring{4}32)$, $(1\mathring{3}\mathring{2}41)$, $(1\mathring{3}\mathring{2}4)$, $(14\|23)$, $(1\mathring{4}\mathring{3}2)$, $(\mathring{1}\mathring{2}43\mathring{5}\mathring{1}) - u_{3}$, $(1\mathring{2}43\mathring{5}1) - u_{3}$, and $(1\mathring{2}43\mathring{5}1) - u_{1}$.
\end{proposition}
\begin{proof}
  Consider first that $u_{3} u_{4}$ is in $E(G)$.  We argue that neither of $u_{1} v_{1}$ and $u_{1} u_{3}$ cannot be present.
  If $u_{1} v_{1}$ is in $E(G)$, then by Obs.\ref{ob:1} (with $i = 1$), $u_{1} u_{3} \in E(G)$.  As a result, $u_{1} u_{4} \notin E(G)$, as otherwise $\{u_{1}, u_{3}, u_{4}, v_{1}\}$ is a clique.  But then $u_{3} v_{3} \in E(G)$ by Obs.\ref{ob:3} (with $i = 3$), and $d(u_{3}) = 6$, a contradiction.
  Likewise, the existence of $u_{1} u_{3}$ would force $u_{3} v_{3} \in E(G)$ by Obs.\ref{ob:7} (with $i = 2$), then $d(u_{3}) = 6$.  
  Now $u_{1}$ is adjacent to neither of $u_{3}$ and $v_1$, the edge $u_{1} u_{4}$ must be present to avoid $d(u_{1}) > 2$.   Moreover, $u_{4} v_{4} \in E(G)$ by Obs.\ref{ob:7} (with $i = 5$), and then from $d(u_{4}) < 6$ it can be inferred  $u_{2} u_{4} \notin E(G)$.
  If neither of the undecided potential edges, $u_{2} v_{2}$ and $u_{3} v_{3}$, is present, then $G$ is $(1\mathring{4}32)$; if only $u_{2} v_{2}$ is present, then $G$ is isomorphic to $(1\mathring{2}43\mathring{5}1) - u_{3}$; if only $u_{3} v_{3}$ is present, then $G$ is $(1\mathring{4}\mathring{3}2)$; otherwise, both are present, and $G$ is isomorphic to $(\mathring{1}\mathring{2}43\mathring{5}\mathring{1}) - u_{3}$.

  In the rest, $u_{3} u_{4}$ is not in $E(G)$.  We consider the potential edges incident to $u_1$ and $u_4$; note that their degrees are at least three.
  By Obs.\ref{ob:1} (with $i = 1$),  the presence of $u_{1} v_{1}$ implies the existence of $u_{1} u_{3}$; likewise, $u_{4} v_{4}$ implies $u_{2} u_{4} \in E(G)$.
  \begin{itemize}
  \item
    Case 1, $u_{1} v_{1}$ is in $E(G)$.
    Note that if $u_{2} u_{4}$ is in $E(G)$, then $u_{4} v_{4}$ must be in $E(G)$ as well; otherwise $u_{2} v_{2} \in E(G)$ by Obs.\ref{ob:7} (with $i = 3$), contradicting Obs.\ref{ob:10} (with $i = 1$).
    First, if $u_{4} v_{4}\in E(G)$, then by Obs.\ref{ob:10} (with $i = 1$), $u_{2} v_{2} \notin E(G)$.  A symmetric argument implies $u_{3} v_{3} \notin E(G)$.  Note that $u_{1} u_{4} \notin E(G)$, as otherwise $G- \{v_{2}, v_{3}\}$ is isomorphic to $\overline{C_7}$.  Then $G$ is isomorphic to $(1\mathring{2}43\mathring{5}1) - u_{1}$.
    Second, if $u_{1} u_{4}$ is in $E(G)$ but $u_{4} v_{4}$ and $u_{2} u_{4}$ are not, then by Obs.\ref{ob:10} (with $i = 1$), $u_{2} v_{2} \notin E(G)$.  Dependent on whether $u_{3} v_{3}$ is present or not,  $G$ is either $(2\mathring{3}\mathring{1}4)$ or $(23\mathring{1}4)$.
  \item
    Case 2, $u_{4} v_{4}$ is in $E(G)$.  It is symmetric to case 1.
  \item
    Case 3, $u_{1} u_{3}$ is in $E(G)$ but $u_{1} v_{1}$ and $u_{4} v_{4}$ are not.  By Obs.\ref{ob:7} (with $i = 2$), $u_{3} v_{3} \in E(G)$.
  If $u_{2} u_{4}$ is in $E(G)$, then by Obs.\ref{ob:7} (with $i = 3$), $u_{2} v_{2} \in E(G)$.  Dependent on whether $u_{1} u_{4}$ is in $E(G)$ or not, $G$ is either $(1\mathring{3}\mathring{2}4)$ or $(1\mathring{3}\mathring{2}41)$.
  If $u_{1} u_{4}$ is in $E(G)$ but $u_{2} u_{4}$ is not, then $u_{2} v_{2} \notin E(G)$, as otherwise $u_{2} u_{3} u_{1} u_{4} v_{2}$ is a $5$-cycle, on which $v_{4}$ has two non-consecutive neighbors, contradicting Lemma~\ref{lem:connection of u to C}.  Then $G$ is $(2\mathring{3}14)$.
  \item
    Case 4, $u_{2} u_{4}$ is in $E(G)$ but $u_{1} v_{1}$ and $u_{4} v_{4}$ are not.  It is symmetric to case 3.
  \end{itemize}
  Now that all of $u_{1} v_{1}$, $u_{4} v_{4}$, $u_{1} u_{3}$, and $u_{2} u_{4}$ are absent, the edge $u_{1} u_{4}$ must be present to ensure $d(u_{1}) > 2$.
  Then $u_{3} v_{3} \notin E(G)$, as otherwise $u_{2}$ has only one neighbor on the $5$-cycle $u_{1} u_{4} v_{1} u_{3} v_{3}$, contradicting Lemma~\ref{lem:connection of u to C}.  A symmetric argument implies $u_{2} v_{2} \notin E(G)$.  Thus, $G$ is $(14\|23)$.
\end{proof}

By Propositions~\ref{prop:contains-all-outer-edges}--\ref{prop:contain-u2u3}, a degree-bounded core graph of order nine is one of     $(1\mathring{3}\mathring{4}\mathring{2})$, $(1\mathring{3}\mathring{4}2)$, $(\mathring{1}\mathring{3}\mathring{4}2)$, $(\mathring{1}\mathring{2}\mathring{4}3)$, $(\mathring{1}\mathring{2}\mathring{4}\mathring{3}\mathring{1})$, $(1\mathring{2}\mathring{4}3)$, $(1\mathring{3}\mathring{2}4)$, $(1\mathring{3}\mathring{2}41)$, $(2\mathring{3}14)$, $(2\mathring{3}\mathring{1}4)$, $(23\mathring{1}4)$, $(\mathring{2}41\mathring{3})$, $(1\mathring{4}32)$, $(14\|23)$, $(1\mathring{4}\mathring{3}2)$, $(1\mathring{3}\|\mathring{2}4)$, or a proper induced subgraph of a $(3, 3)$-partitionable graph.

\begin{lemma} \label{lem:minimal-graphs-of-pattern-g}
  All degree-bounded core graphs of order nine are t-perfect.  Only $(\mathring{1}2\mathring{3}45\mathring{1}) - u_{2}$, $(123\mathring{4}\mathring{5}1) - u_{2}$, $(1\mathring{2}43\mathring{5}1) - u_{1}$, $(1\mathring{3}\mathring{2}4)$, and $(\mathring{2}41\mathring{3})$ of them are self-complementary graphs.
\end{lemma}

We are now ready to prove Theorem~\ref{thm:self-complement}.
\begin{proof}[Proof of Theorem~\ref{thm:self-complement}]
  The sufficiency is quite obvious.  One may easily verify that $C_{5}$, $(\mathring{2}41\mathring{3})$, $(1\mathring{3}\mathring{2}4)$, $(1\mathring{2}\mathring{3}4)$, $(\mathring{1}23\mathring{4})$, and $(\mathring{1}32\mathring{4})$ are all self-complementary.  Since all of them contain a $C_5$, they are not perfect.
We have seen that  $(\mathring{2}41\mathring{3})$ and $(1\mathring{3}\mathring{2}4)$ are t-perfect; on the other hand, $(1\mathring{2}\mathring{3}4)$, $(\mathring{1}23\mathring{4})$, and $(\mathring{1}32\mathring{4})$ are isomorphic to  $(123\mathring{4}\mathring{5}1) - u_{2}$, $(\mathring{1}2\mathring{3}45\mathring{1}) - u_{2}$, and $(1\mathring{2}43\mathring{5}1) - u_{1}$ respectively, hence t-perfect as well.

  For the necessity, suppose that $G$ is a self-complementary t-perfect graph and not perfect.  Since both $G$ and $\overline{G}$ are t-perfect, $G$ is a core graph.  By Corollary \ref{cor:ten vertices}, $5 \le n\le 10$.  Since the order of a self-complementary graph is either $4k$ or $4k + 1$ for some $k \ge 0$, we can have $n \in \{5,8,9\}$.  Since $G$ is not perfect, it contains an odd hole, and by Proposition~\ref{lem:odd hole constrain}, every odd hole in $G$ is a $5$-cycle.
  If $n = 5$, then $G$ is $C_{5}$.

  If $n = 9$, then $G$ is of pattern Fig.~\ref{fig:possible configurations of G}(b).  We argue that $G$ is degree bounded.  Every vertex in $C$ has degree at least three and at most five.  Suppose that one vertex $u\in U$ has degree two, then it is not adjacent to any other vertex in $U$.  But then the degree of $u$ in $\overline G$ is six; thus there is a degree-6 vertex, which has to be in $U$.  But then we have a vertex in $U$ that is nonadjacent to others in $U$, and another vertex in $U$ that is adjacent to all of the others in $U$, a contradiction.  By Lem.~\ref{lem:minimal-graphs-of-pattern-g}, $G$ is one of $(\mathring{1}23\mathring{4})$, $(\mathring{2}41\mathring{3})$, $(\mathring{1}32\mathring{4})$, $(1\mathring{2}\mathring{3}4)$, and $(1\mathring{3}\mathring{2}4)$.

  It remains to show that there is no graph of order 8 satisfying the conditions.  Let $G$ be a core graph of order $8$.  We may assume that the indices for the three vertices in $U$ are not consecutive: If $G$ is of pattern Figure~\ref{fig:possible configurations of G}(a), then we can consider its complement.  (With different choices of $5$-cycles, a core graph may be of more than one patterns.)  If there is a vertex $x$ of degree $2$, then $x\in U$, and the two neighbors of $x$ are adjacent.  Then in $\overline{G}$, every vertex in $U$ has degree at least three, which means $x$ is mapped to a vertex $y$ in $C$.  However, if $y$ has degree two, then its two neighbors are not adjacent in $\overline{G}$, a contradiction.  Therefore, the minimum degree is at least three, and since $G$ is self-complementary, the maximum degree is at most four.
  By Lemma~\ref{lem:t-perfect-order-8}, $G$ can only be one of $(\mathring{1}\mathring{3}\|\mathring{2})$, $(\mathring{1}3\|\mathring{2})$, $(\mathring{1}2\mathring{3})$, $(\mathring{3}1\mathring{2})$, $(\mathring{1}23)$, $(\mathring{3}12)$, $\overline{(\mathring{1}\|\mathring{2}\|\mathring{4})}$, and $(123)$, but none of them is self-complementary.
\end{proof}

\section{Proof of Theorem~\ref{thm:main}}
\label{sec:mti}

Bruhn and Stein~\cite{bruhn-10-t-perfection-claw-free} showed that the $(3,3)$-partitionable graphs are minimally t-imperfect.\footnote{They showed that those containing $C_{10}^{2}$ are minimally t-imperfect, while $(1\mathring{2}43\mathring{5}1)$ and $(\mathring{1}\mathring{2}43\mathring{5}\mathring{1})$, which are complements to each other, are referred to an unpublished manuscript or Bruhn.  For the sake of completeness, we provide a proof at the appendix.}  Therefore, we only need to show the sufficiency in Theorem~\ref{thm:main}.
We say that a clique $K$ of a connected graph $G$ is a \emph{clique separator} of $G$ if $G-K$ is not connected.
\begin{lemma}[Chv{\'{a}}tal~\cite{chvatal-75-graph-polytopes}, Gerards~\cite{gerards-98-all-subgraphs-t-perfect}]
  \label{clique separator}
  No minimally t-imperfect graph contains a clique separator.
\end{lemma}

Throughout this section, we assume that both $G$ and its complement $\overline{G}$ are minimally t-imperfect graphs.  By Lemma~\ref{clique separator}, neither $G$ nor $\overline{G}$ can have a clique separator.   Thus, for each vertex $u \in U$, we have
\begin{equation}
  \label{eq:1}
  2 < d(u) < n - 3.
\end{equation}
Note that if $d(u) = n - 3$, then $u$ has two neighbors in $\overline{G}$, which is a clique separator.

Note that $G$ is a core graph.   By Proposition~\ref{lem:odd hole constrain} and Corollary \ref{cor:ten vertices}, the order of $G$ is between five and ten.
Recall that every almost bipartite graph, which contains a vertex whose deletion leaves the graph bipartite, is t-perfect~\cite{fonlupt-82-h-perfectness}.
The only core graph of order five is $C_5$.  Both core graphs of order six, $(1)$ and $(\mathring{1})$, are almost bipartite, e.g., removing $v_3$.  There are 16 core graphs of order seven, $(1 2)$, $(\mathring{1} 2)$, $(1 \mathring{2})$, $(\mathring{1} \mathring{2})$, $(1\| 2)$, $(\mathring{1}\| 2)$, $(1\| \mathring{2})$, $(\mathring{1} \|\mathring{2})$, and their complements.  All the listed eight graphs become bipartite after removing $v_4$, hence almost bipartite.
By Lemma~\ref{lem:t-perfect-order-8} and the degree requirements \eqref{eq:1}, $G$ cannot have order eight either.
Likewise, by Lemma~\ref{lem:minimal-graphs-of-pattern-g}, all core graphs of order nine satisfying \eqref{eq:1} are t-perfect.  Therefore, we are only left with $n = 10$.

In the rest of this section, the order of $G$ is ten.  Our analysis is based on whether $(123451)$ is a (not necessarily induced) subgraph of $G$.  The arguments here are somewhat similar to that in Section~\ref{sec:graph-with-pattern-g}.  Let us start with an easy case, where all the five edges $u_iu_{i+1}$ for $i = 1, \ldots, 5$ are in $E(G)$; see Figure~\ref{fig:10 vertices}(a).  Recall that all the indices are understood as modulo 5.

\begin{figure} [h]
  \centering
  \begin{subfigure}[b]{.24\linewidth}
    \centering
    \begin{tikzpicture}[scale=.75]
      \foreach \i in {1,..., 5} {
          \draw[normal edge] ({54 + \i * (360 / 5)}:1.25) -- ({126 + \i * (360 / 5)}:1.25);
        }
      \foreach \i in {1,..., 5} {
          \node[filled vertex] (v\i) at ({54 + \i * (360 / 5)}:1.25) {};
          \node at ({54 + \i * (360 / 5)}:1.5) {$v_{\i}$};
        }
      \foreach \i in {1,..., 5} {
          \node[filled vertex] (u\i) at ({234 + \i * (360 / 5)}:2) {};
          \node at ({234 + \i * (360 / 5)}:2.25) {$u_{\i}$};
        }
      \draw [normal edge] (u1) -- (v3) (u1) -- (v4);
      \draw [normal edge] (u2) -- (v4) (u2) -- (v5);
      \draw [normal edge] (u3) -- (v5) (u3) -- (v1);
      \draw [normal edge] (u4) -- (v1) (u4) -- (v2);
      \draw [normal edge] (u5) -- (v2) (u5) -- (v3);
      \draw [normal edge] (u1) -- (u5) (u1) -- (u2) (u2) -- (u3) (u3) -- (u4) (u4) -- (u5);
      \draw [potential edge]  (u1) -- (v1)  (u2) -- (v2) (u3) -- (v3) (u4) -- (v4) (u5) -- (v5);
      \draw [potential edge]  (u1) edge [bend right] (u3) (u1) edge [bend left] (u4) (u2) edge [bend right] (u4) (u2) edge [bend left] (u5) (u3) edge [bend right] (u5);
    \end{tikzpicture}
    \caption{}
  \end{subfigure}
  \begin{subfigure}[b]{.24\linewidth}
    \centering
    \begin{tikzpicture}[scale=.75]
      \foreach \i in {1,..., 5} {
          \draw[normal edge] ({54 + \i * (360 / 5)}:1.25) -- ({126 + \i * (360 / 5)}:1.25);
        }
      \foreach \i in {1,..., 5} {
          \node[filled vertex] (v\i) at ({54 + \i * (360 / 5)}:1.25) {};
          \node at ({54 + \i * (360 / 5)}:1.5) {$v_{\i}$};
        }
      \foreach \i in {1,..., 5} {
          \node[filled vertex] (u\i) at ({234 + \i * (360 / 5)}:2) {};
          \node at ({234 + \i * (360 / 5)}:2.25) {$u_{\i}$};
        }
      \draw [normal edge] (u1) -- (v3) (u1) -- (v4);
      \draw [normal edge] (u2) -- (v4) (u2) -- (v5);
      \draw [normal edge] (u3) -- (v5) (u3) -- (v1);
      \draw [normal edge] (u4) -- (v1) (u4) -- (v2);
      \draw [normal edge] (u5) -- (v2) (u5) -- (v3);
      \draw [normal edge] (u2) edge [bend right] (u4);
      \draw [potential edge]  (u2) edge [bend left] (u5) (u1) edge [bend right] (u3) (u3) edge [bend right] (u5)  (u1) edge [bend left] (u4);
      \draw [potential edge]  (u1) -- (v1)  (u2) -- (v2) (u3) -- (v3) (u4) -- (v4) (u5) -- (v5);
      \draw [forbidden edge] (u2) -- (u3) (u3) -- (u4);
      \draw [potential edge]  (u4) -- (u5) (u1) -- (u5) (u1) -- (u2);
    \end{tikzpicture}
    \caption{}
    \label{}
  \end{subfigure}
  \begin{subfigure}[b]{.24\linewidth}
    \centering
    \begin{tikzpicture}[scale=.75]
      \foreach \i in {1,..., 5} {
          \draw[normal edge] ({54 + \i * (360 / 5)}:1.25) -- ({126 + \i * (360 / 5)}:1.25);
        }
      \foreach \i in {1,..., 5} {
          \node[filled vertex] (v\i) at ({54 + \i * (360 / 5)}:1.25) {};
          \node at ({54 + \i * (360 / 5)}:1.5) {$v_{\i}$};
        }
      \foreach \i in {1,..., 5} {
          \node[filled vertex] (u\i) at ({234 + \i * (360 / 5)}:2) {};
          \node at ({234 + \i * (360 / 5)}:2.25) {$u_{\i}$};
        }
      \draw [normal edge] (u1) -- (v3) (u1) -- (v4);
      \draw [normal edge] (u2) -- (v4) (u2) -- (v5);
      \draw [normal edge] (u3) -- (v5) (u3) -- (v1);
      \draw [normal edge] (u4) -- (v1) (u4) -- (v2);
      \draw [normal edge] (u5) -- (v2) (u5) -- (v3);
      \draw [potential edge]  (u2) edge [bend left] (u5) (u1) edge [bend right] (u3) (u3) edge [bend right] (u5)  (u1) edge [bend left] (u4);
      \draw [potential edge]  (u1) -- (v1)  (u2) -- (v2) (u3) -- (v3) (u4) -- (v4) (u5) -- (v5);
      \draw [forbidden edge] (u2) -- (u3) (u3) -- (u4) (u2) edge [bend right] (u4);
      \draw [potential edge]  (u4) -- (u5) (u1) -- (u5) (u1) -- (u2);
    \end{tikzpicture}
    \caption{}
  \end{subfigure}
  \begin{subfigure}[b]{.24\linewidth}
    \centering
    \begin{tikzpicture}[scale=.75]
      \foreach \i in {1,..., 5} {
          \draw[normal edge] ({54 + \i * (360 / 5)}:1.25) -- ({126 + \i * (360 / 5)}:1.25);
        }
      \foreach \i in {1,..., 5} {
          \node[filled vertex] (v\i) at ({54 + \i * (360 / 5)}:1.25) {};
          \node at ({54 + \i * (360 / 5)}:1.5) {$v_{\i}$};
        }
      \foreach \i in {1,..., 5} {
          \node[filled vertex] (u\i) at ({234 + \i * (360 / 5)}:2) {};
          \node at ({234 + \i * (360 / 5)}:2.25) {$u_{\i}$};
        }
      \draw [normal edge] (u1) -- (v3) (u1) -- (v4);
      \draw [normal edge] (u2) -- (v4) (u2) -- (v5);
      \draw [normal edge] (u3) -- (v5) (u3) -- (v1);
      \draw [normal edge] (u4) -- (v1) (u4) -- (v2);
      \draw [normal edge] (u5) -- (v2) (u5) -- (v3);
      \draw [potential edge] (u2) edge [bend left] (u5) (u1) edge [bend right] (u3) (u3) edge [bend right] (u5) (u1) edge [bend left] (u4);
      \draw [potential edge] (u1) -- (v1)  (u2) -- (v2) (u3) -- (v3) (u4) -- (v4) (u5) -- (v5) (u4) -- (u5) (u2) edge [bend right] (u4) (u3) -- (u4);
      \draw [forbidden edge] (u1) -- (u2);
      \draw [normal edge] (u2) -- (u3) (u1) -- (u5);
    \end{tikzpicture}
    \caption{}
  \end{subfigure}
  \caption{(a) All the five edges $u_iu_{i+1}$ for $i = 1, \ldots, 5$ are present; (b) both $u_{2}u_{3}$ and $u_{3}u_{4}$ are absent while $u_2 u_4$ is present; (c) all the edges among $u_{2}, u_{3}, u_{4}$ are absent; (d) $u_1 u_2$ is absent, while only $u_{2} u_{3}$ and $u_{1}u_{5}$ are present.}
  \label{fig:10 vertices}
\end{figure}
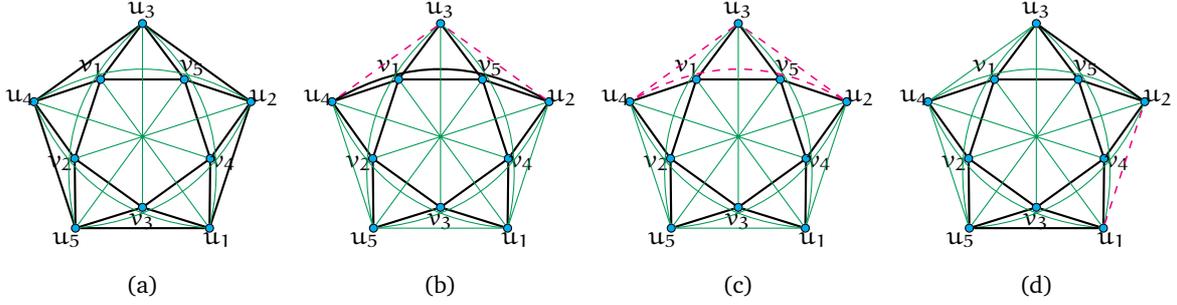

\begin{proposition} \label{prop:contains-all}
  If for all $i = 1, \ldots, 5$, the edge $u_i u_{i+1}$ is in $E(G)$, then $G$ is one of the $(3,3)$-partitionable graphs.
\end{proposition}
\begin{proof}
  We first argue that $U$ induces a cycle.  Suppose for contradiction that $u_{1} u_{3}$ is present.  By Obs.\ref{ob:7} (with $i=2$), at least one of $u_{1} v_{1}$ and $u_{3} v_{3}$ is in $E(G)$.  Since they are symmetric, we consider $u_{1} v_{1} \in E(G)$.   Since $\{u_{1}, u_{3}, u_{4}, v_{1}\}$ is not a clique, $u_{1} u_{4} \notin E(G)$.
  Then by Obs.\ref{ob:3} (with $i=3$), $u_{3} v_{3} \in E(G)$, and since $\{u_{1}, u_{3}, u_{5}, v_{3}\}$ is not a clique, $u_{3} u_{5} \notin E(G)$.
  But then $G- \{v_{4}, v_{5}, u_{2}\}$ is isomorphic to $\overline{C_{7}}$, and $G$ is not minimally t-imperfect.

  Now that $G[U]$ is a $C_5$, dependent on the combination of edges $u_i v_i, i = 1, \ldots, 5$, we are in one of the $(3,3)$-partitionable graphs that contain $(123451)$.
\end{proof}

The following two propositions deal with the case where for some $i = 1, \ldots, 5$, both edges $u_{i}u_{i-1}$ and $u_{i}u_{i+1}$ are absent, Proposition~\ref{prop:missing-consecutive-two} for $u_{i-1} u_{i+1}$ being present, and Proposition~\ref{lem:missing-consecutive-two} for otherwise; see Figure~\ref{fig:10 vertices}(b, c).

\begin{proposition}\label{prop:missing-consecutive-two}
  Let $i = 1, \ldots, 5$.  If neither $u_{i} u_{i-1}$ nor $u_{i} u_{i+1}$ is in $E(G)$, then $u_{i-1} u_{i+1}$ cannot be in $E(G)$ either.
\end{proposition}
\begin{proof}
  Assume without loss of generality $i=3$; i.e., both $u_2 u_3$ and $u_3 u_4$ are absent, and  we show by contradiction that $u_{2} u_{4}$ cannot be in $E(G)$.   By Obs.\ref{ob:7} (with $i = 3$), at least one of $u_{2} v_{2}$ and $u_{4} v_{4}$ is in $E(G)$.  Since they are symmetric, we may consider $u_{2} v_{2}\in E(G)$.  %

  Suppose that $u_{4} u_{5} \in E(G)$.  Then $u_{2} u_{5} \notin E(G)$, as otherwise, $\{u_{2}, u_{4}, u_{5}, v_{2}\}$ is a $K_{4}$.  By Obs.\ref{ob:3} (with $i=4$), $u_{4} v_{4} \in E(G)$.  If $u_{1} u_{2} \in E(G)$, then $u_{1} u_{4} \notin E(G)$ because $\{u_{2}, u_{1}, u_{4}, v_{4}\}$ cannot be  a $K_{4}$; by Obs.\ref{ob:5} (with $i=3$), $u_{1} u_{5} \in E(G)$, but then $G- \{u_{3}, v_{1}, v_{5}\}$ is a $\overline{C_{7}}$, a contradiction to Proposition~\ref{lem:odd hole constrain}.  Now that $u_{1} u_{2} \notin E(G)$. The edge $u_{1} u_{5}$ is not in $E(G)$, as otherwise $u_{2} v_{2}$ is in $E(G)$ and both $u_{1} u_{2}$ and $u_{2} u_{5}$ are not, a contradiction to Obs.\ref{ob:1} (with $i = 2$). By Obs.\ref{ob:10} (with $i=1$), $u_{1} v_{1}$ cannot be in $E(G)$ either.  The set $\{u_{2},u_{1},u_{5},v_{1}\}$ is an independent set in $G$, a contradiction.  In the rest, $u_{4} u_{5} \notin E(G)$.

  Suppose $u_{3} u_{5} \in E(G)$.  By Obs.\ref{ob:10} (with $i = 2$), $u_{3} v_{3} \notin E(G)$.  Further, By Obs.\ref{ob:12}, ~\ref{ob:7}, and~\ref{ob:10} (with $i = 1$, $i = 4$, and $i = 4$ respectively), $u_{2} u_{5} \notin E(G)$, $u_{5} v_{5} \in E(G)$, and $u_{4} v_{4} \notin E(G)$.  Hence, $u_{2} v_{4} v_{3} u_{5} u_{3} v_{1} u_{4}$ is a $7$-cycle in $G$, contradicting Proposition~\ref{lem:odd hole constrain}.  Thus, $u_{3} u_{5} \notin E(G)$.

  Suppose $u_{4} v_{4} \in E(G)$.  By Obs.\ref{ob:10} (with $i = 4$), $u_{5} v_{5} \notin E(G)$.  Since $\{u_{1},u_{5},u_{4},v_{5}\}$  cannot be an independent set, at least one of $u_{1} u_{4}$ and $u_{1} u_{5}$ needs to be present.
  By Obs.\ref{ob:1} (with $i = 4$), if $u_{1} u_{5} \in E(G)$, then $u_{1} u_{4} \in E(G)$ as well.
  Therefore, we always have $u_{1} u_{4} \in E(G)$.  Since $\{u_{1}, u_{2}, u_{4}, v_{4}\}$ cannot induce a $K_{4}$ in $G$, $u_{1} u_{2} \notin E(G)$.  By Obs.\ref{ob:10} (with $i=1$), $u_{1} v_{1} \notin E(G)$.  Since $\{u_{2},u_{1},u_{5},v_{1}\}$ is not an independent set, at least one of $u_{1} u_{5}$ and $u_{2} u_{5}$ is in $E(G)$.  If $u_{1} u_{5}$ is in $E(G)$, then by Obs.\ref{ob:1} (with $i = 2$), $u_{2} u_{5} \in E(G)$ and $G- \{u_{3}, v_{1}, v_{5}\}$ is isomorphic to $\overline{C_{7}}$.  Otherwise, by Obs.\ref{ob:12} (with $i = 3$), $u_{2} u_{5}$ has to be absent as well, and then $\{u_{2},u_{1},u_{5},v_{1}\}$ is an independent set.

  Therefore, none of $u_{3} u_{5}$, $u_{4} u_{5}$, and $u_{4} v_{4}$ can be in $E(G)$, and then $\{u_{3},u_{4},u_{5},v_{4}\}$ forms an independent set, contradicting  Proposition~\ref{lem:K4 and I4 free}.
\end{proof}

\begin{proposition}\label{lem:missing-consecutive-two}
  For all $i = 1, \ldots, 5$, at least one of $u_{i} u_{i-1}$ and $u_{i} u_{i+1}$ is in $E(G)$.
\end{proposition}
\begin{proof}
  Assume without loss of generality, let $i=3$.  Suppose for contradiction that neither $u_{2} u_{3}$ nor $u_{3} u_{4}$ is in $E(G)$.  By Proposition~\ref{prop:missing-consecutive-two}, $u_{2} u_{4} \notin E(G)$.  Thus, $u_{3} v_{3}\in E(G)$, as otherwise $\{u_{2},u_{3},u_{4},v_{3}\}$ forms an independent set.
  As a result, $u_{2} v_{2}\not\in E(G)$, as otherwise $u_{4}$ has only one neighbor on the $5$-cycle $u_{3} v_{3} v_{2} u_{2} v_{5}$.
  Moreover, $u_{1} u_{2}$ must be in $G$: Otherwise, by Proposition~\ref{prop:missing-consecutive-two}, (noting that $u_{2} u_{3}\not\in E(G)$,) $u_{1} u_{3}$ cannot be in $E(G)$ either, then $\{u_{1},u_{2},u_{3},v_{2}\}$ forms an independent set.
  By Obs.\ref{ob:1} (with $i = 3$), $u_{1} u_{3} \in E(G)$, and then by Obs.\ref{ob:10} (with $i = 3$), $u_{4} v_{4} \notin E(G)$.   Since $\{u_{1},u_{5},u_{4},v_{5}\}$ does not induce an independent set, at least one of $u_{1} u_{5}$, $u_{4} u_{5}$, $u_{5} v_{5}$, and $u_{1} u_{4}$ is in $E(G)$.

  First, suppose that $u_{1} u_{5}$ is in $E(G)$.   Then $u_{3} u_{5}$ is not in $E(G)$, as otherwise $\{u_{3}, u_{5}, u_{1}, v_{3}\}$ induces a $K_{4}$.   By Obs.\ref{ob:3} (with $i=1$), $u_{1} v_{1} \in E(G)$. The edge $u_{4} u_{5} \notin E(G)$, as otherwise contracting Obs.\ref{ob:1} (with $i = 3$). But then $\{u_{3},u_{4},u_{5},v_{4}\}$ forms an independent set.

  Second, suppose that $u_{4} u_{5}$ is in $E(G)$.   By Obs.\ref{ob:1} (with $i = 3$), $u_{3} u_{5} \in E(G)$.
  If $u_{1} v_{1}$ is in $E(G)$, then by Obs.\ref{ob:1} (with $i=1$), $u_{1} u_{4} \in E(G)$, which means that $G- \{u_{2}, v_{4}, v_{5}\}$ is isomorphic to $\overline{C_{7}}$.   Thus, $u_{1} v_{1} \notin E(G)$; a symmetric argument enables us to conclude that $u_{5} v_{5} \notin E(G)$.
  Since neither of $u_{2} u_{4}$ and $u_{5} v_{5}$ is in $E(G)$, from Obs.\ref{ob:3} (with $i = 5$) we can conclude that, $u_{2} u_{5} \notin E(G)$.   By a symmetric argument we have $u_{1} u_{4}$ is not in $E(G)$ either.
  Now that none of $u_{1} v_{1}$, $u_{5} v_{5}$, $u_{2} u_{5}$, and $u_{1} u_{4}$ is in $E(G)$, there is a $7$-cycle $u_{5} u_{4} v_{1} v_{5} u_{2} u_{1} v_{3}$.   Therefore, $u_{4} u_{5} \notin E(G)$.

  Third, suppose $u_{5} v_{5}$ is in $E(G)$.   By Obs.\ref{ob:1} (with $i = 5$), $u_{2} u_{5} \in E(G)$.  The edge $u_{3} u_{5} \in E(G)$, as otherwise $\{u_{3},u_{4},u_{5},v_{4}\}$ forms an independent set. But then $G- \{v_{1}, v_{2}, u_{4}\}$ is isomorphic to $\overline{C_{7}}$.   Therefore, $u_{5} v_{5} \notin E(G)$.

  Last, suppose $u_{1} u_{4}$ is in $E(G)$.   By Obs.\ref{ob:12} (with $i = 2$), $u_{3} u_{5} \notin E(G)$.   But then $\{u_{3},u_{4},u_{5},v_{4}\}$ forms an independent set.

  In summary, none of $u_{1} u_{5}$, $u_{4} u_{5}$, $u_{5} v_{5}$, and $u_{1} u_{4}$ can be in $E(G)$, and thus $\{u_{1},u_{5},u_{4},v_{5}\}$ forms an independent set.
\end{proof}

In the remaining case, $u_{i} u_{i+1}$ for some $i = 1, \ldots, 5$ is absent, but both $u_{i+1} u_{i+2}$ and $u_{i} u_{i-1}$ are present.  Moreover, by Proposition~\ref{lem:missing-consecutive-two}, at least one of $u_{i+2} u_{i+3}$ and $u_{i-1} u_{i-2}$ is in $E(G)$.     See Figure~\ref{fig:10 vertices}(d).

\begin{proposition}\label{lem:one-missing}
  If there is an $i = 1, \ldots, 5$ such that $u_iu_{i+1}$ is not in $E(G)$, then $G$ is one of the $(3, 3)$-graphs.
\end{proposition}
\begin{proof}
  Without loss of generality, let $i=1$.  Then $u_{1} u_{2} \notin E(G)$, $u_{2} u_{3}$ and $u_{1} u_{5}$ are in $E(G)$, and at least one of $u_{3} u_{4}$ and $u_{4} u_{5}$ is in $E(G)$.  We show by contradiction that $u_{3} u_{4}$ and $u_{4} u_{5}$ cannot be both in $E(G)$.   In particular, we show that
  none of $u_{1} v_{1}$, $u_{2} v_{2}$, $u_{1} u_{3}$, $u_{2} u_{5}$, and $u_{3} u_{5}$ is in $E(G)$, and then $u_{1} v_{4} u_{2} u_{3} v_{1} v_{2} u_{5}$ is a ${7}$-cycle.
  \begin{itemize}
    \item If $u_{3} u_{5}$ is in $E(G)$, then by Obs.\ref{ob:7} (with $i=4$), at least one of $u_{3} v_{3}$ and $u_{5} v_{5}$ is in $E(G)$.
          If $u_{3} v_{3}$ is in $E(G)$ but $u_{5} v_{5}$ is not, then by Obs.\ref{ob:3} (with $i=5$), $u_{1} u_{3} \in E(G)$, which means $d(u_{3}) = 7$, a contradiction.
          A symmetric argument applies if $u_{5} v_{5}$ is in $E(G)$ but $u_{3} v_{3}$ is not.
          Hence, both $u_{3} v_{3}$ and $u_{5} v_{5}$ are in $E(G)$.   As a result, neither $u_{1} u_{3}$ nor $u_{2} u_{5}$ can be in $E(G)$, as otherwise $\{u_{3},u_{1},u_{5},v_{3}\}$ or, respectively, $\{u_{5},u_{2},u_{3},v_{5}\}$ forms a clique.   However, the vertex $v_{3}$ has four neighbors on a $5$-cycle $u_{3} u_{5} u_{1} v_{4} u_{2}$.   Therefore, $u_{3} u_{5} \notin E(G)$.

    \item If $u_{1} v_{1}$ is in $E(G)$, then by Obs.\ref{ob:1} (with $i=1$), $u_{1} u_{3} \in E(G)$.   Note that $u_{1} u_{4} \notin E(G)$, as otherwise $\{u_{1},u_{3},u_{4},v_{1}\}$ forms a cliqued.   By Obs.\ref{ob:3} (with $i=3$), $u_{3} v_{3} \in E(G)$.   But then $G- \{v_{4}, v_{5}, u_{2}\}$ is isomorphic to $\overline{C_{7}}$.   Therefore, $u_{1} v_{1} \notin E(G)$.     By a symmetric argument, $u_{2} v_{2} \notin E(G)$.

    \item Now that none of $u_{1} v_{1}$, $u_{2} v_{2}$, and $u_{3} u_{5}$ is in $E(G)$, from Obs.\ref{ob:3} (with $i=1$) it can be inferred $u_{1} u_{3} \notin E(G)$, and then by Obs.\ref{ob:3} (with $i=2$), $u_{2} u_{5} \notin E(G)$.
  \end{itemize}

  Thus, at most one of $u_{3} u_{4}$ and $u_{4} u_{5}$ is in $E(G)$.   We may assume without loss of generality that $u_{3} u_{4}$ is in $E(G)$ and $u_{4} u_{5}$ is not; the other case is symmetric.

  We argue that none of $u_{1} u_{3}$, $u_{3} u_{5}$, $u_{1} v_{1}$, and $u_{5} v_{5}$ can be in $E(G)$.
  Suppose that $u_{1} u_{3}$ is in $E(G)$.
  By Obs.\ref{ob:3} (with $i=1$), at least one of $u_{1} v_{1}$ and $u_{3} u_{5}$ is in $E(G)$.   If $u_{3} u_{5}\in E(G)$, then $u_{3} v_{3} \notin E(G)$, as otherwise $\{u_{3},u_{1},u_{5},v_{3}\}$ forms a clique.
  On the other hand, by Obs.\ref{ob:7} (with $i=2$), at least one of $u_{1} v_{1}$ and $u_{3} v_{3}$ is in $E(G)$.   Therefore, we always have $u_{1} v_{1} \in E(G)$.
  Then $u_{1} u_{4} \notin E(G)$, as otherwise $\{u_{1},u_{3},u_{4},v_{1}\}$ forms a clique.
By Obs.\ref{ob:3} (with $i = 3$), $u_{3} v_{3} \in E(G)$, which further implies $u_{3} u_{5}\not\in E(G)$ because $d(u_{3}) < 6$. But then all of $u_{1} u_{5}$, $u_{1} u_{3}$, $u_{3} u_{4}$, and $u_{3} v_{3}$ are in $E(G)$ and none of $u_{1} u_{4}$, $u_{3} u_{5}$, and $u_{4} u_{5}$ is in $E(G)$, contradicting Obs.\ref{ob:5} (with $i = 2$).
  Therefore, $u_{1} u_{3} \notin E(G)$.
  By a symmetric argument, we can conclude that $u_{3} u_{5}$ cannot be in $E(G)$ either.
  Now that none of $u_{1} u_{3}$, $u_{3} u_{5}$, $u_{1} u_{2}$, and $u_{4} u_{5}$ is in $E(G)$, together with the fact that both $u_{2} u_{3}$ and $u_{3} u_{4}$ are in $E(G)$, from Obs.\ref{ob:1} (with $i=1$ and $i=5$), it can be inferred that both $u_{1} v_{1}$ and $u_{5} v_{5}$ cannot be in $E(G)$.

  At least one of $u_{4} v_{4}$ and $u_{1} u_{4}$ is in $E(G)$, as otherwise $u_{1} v_{4} v_{5} u_{3} u_{4} v_{2} u_{5}$ is a $7$-cycle.
  If $u_{4} v_{4}$ is in $E(G)$, then Obs.\ref{ob:1} (with $i=4$) will force $u_{1} u_{4}$ in $E(G)$ as well.   On the other hand,
  $u_{1} u_{4}$ is in $E(G)$ and Obs.\ref{ob:7} (with $i=5$) will force $u_{4} v_{4}$ in $E(G)$ as well.  Therefore, both $u_{4} v_{4}$ and $u_{1} u_{4}$ are in $E(G)$.  Moreover, at least one of $u_{2} v_{2}$ and $u_{2} u_{5}$ is in $E(G)$, as otherwise $u_{5} v_{2} v_{1} u_{3} u_{2} v_{4} u_{1}$ is a $7$-cycle.  By a symmetric argument, both $u_{2} v_{2}$ and $u_{2} u_{5}$ are in $E(G)$.  Note that $u_{2} u_{4}$ cannot be in $E(G)$, as otherwise $G- \{v_{1}, v_{5}, u_{3}\}$ is isomorphic to $\overline{C_{7}}$.
  Dependent on whether $u_{3} v_{3}$ is in $E(G)$, the graph is isomorphic to either $(1\mathring{2}43\mathring{5}1)$ or its complement.
\end{proof}

The discussion on the order of $G$ and Propositions \ref{prop:contains-all}--\ref{lem:one-missing} imply Theorem~\ref{thm:main}.

\bibliographystyle{plainurl}

\newpage\appendix
\section*{Appendix: Omitted Proofs} 
\label{sec:t-perfect-graphs}

The following sufficient condition for t-perfection is due to Benchetrit~\cite{benchetrit-15-thesis}.

\begin{proposition}[\cite{benchetrit-15-thesis}]
  \label{prop:nice graph}
  Let $K$ be a clique of a graph $G$.  If $G-v$ is t-perfect for every $v \in K$, then $G$ is t-perfect.
\end{proposition}

\paragraph{Proposition~\ref{lem:t-perfect-graphs} (restated).}
  \textit{The following graphs are t-perfect: $(12)$, $(1\|\mathring{2})$, $(1\mathring{2})$, $(\mathring{1}\|\mathring{2})$, $(\mathring{1}\mathring{2})$, $(1\|23)$, $(\mathring{3}1\mathring{2})$, $(\mathring{1}\mathring{3}\|\mathring{2})$, $(1\|\mathring{2}4)$, $(14\mathring{2})$, $(1\|\mathring{2}\mathring{4})$, $(1\mathring{4}\mathring{2})$, $(\mathring{1}\|\mathring{2}\|\mathring{4})$, $(\mathring{1}\mathring{2}4)$, $(\mathring{1}\mathring{2}\mathring{4})$, $(1\mathring{3}\mathring{4}\mathring{2})$, $(1\mathring{3}\mathring{4}2)$, $(\mathring{1}\mathring{3}\mathring{4}2)$, $(\mathring{1}\mathring{2}\mathring{4}3)$, $(2\mathring{3}14)$, $(2\mathring{3}\mathring{1}4)$, $(23\mathring{1}4)$, $(1\mathring{4}32)$, $(\mathring{1}\mathring{2}\mathring{4}\mathring{3}\mathring{1})$, $(1\mathring{2}\mathring{4}3)$, $(1\mathring{3}\mathring{2}41)$, $(1\mathring{3}\mathring{2}4)$, $(14\|23)$, $(1\mathring{4}\mathring{3}2)$, $(1\mathring{3}\|\mathring{2}4)$, and $(\mathring{2}41\mathring{3})$.}
\begin{proof}
  Graphs $(12)$, $(1\|\mathring{2})$, $(1\mathring{2})$, $(\mathring{1}\|\mathring{2})$, and $(\mathring{1}\mathring{2})$ are almost bipartite graph, hence t-perfect.
  For each of the other graphs, we find a $3$-clique $K$ and then use Proposition~\ref{prop:nice graph}.  To show $G-v$ is t-perfect for every $v \in K$, we either directly show that it is isomorphic to a t-perfect graph, or show that it is a $K_4$-free perfect graph (Proposition~\ref{prop:perfect and t-perfect}).  The details are listed in Table \ref{table:2}, where $\star$ means that the graph is a $K_4$-free perfect graph.
\end{proof}

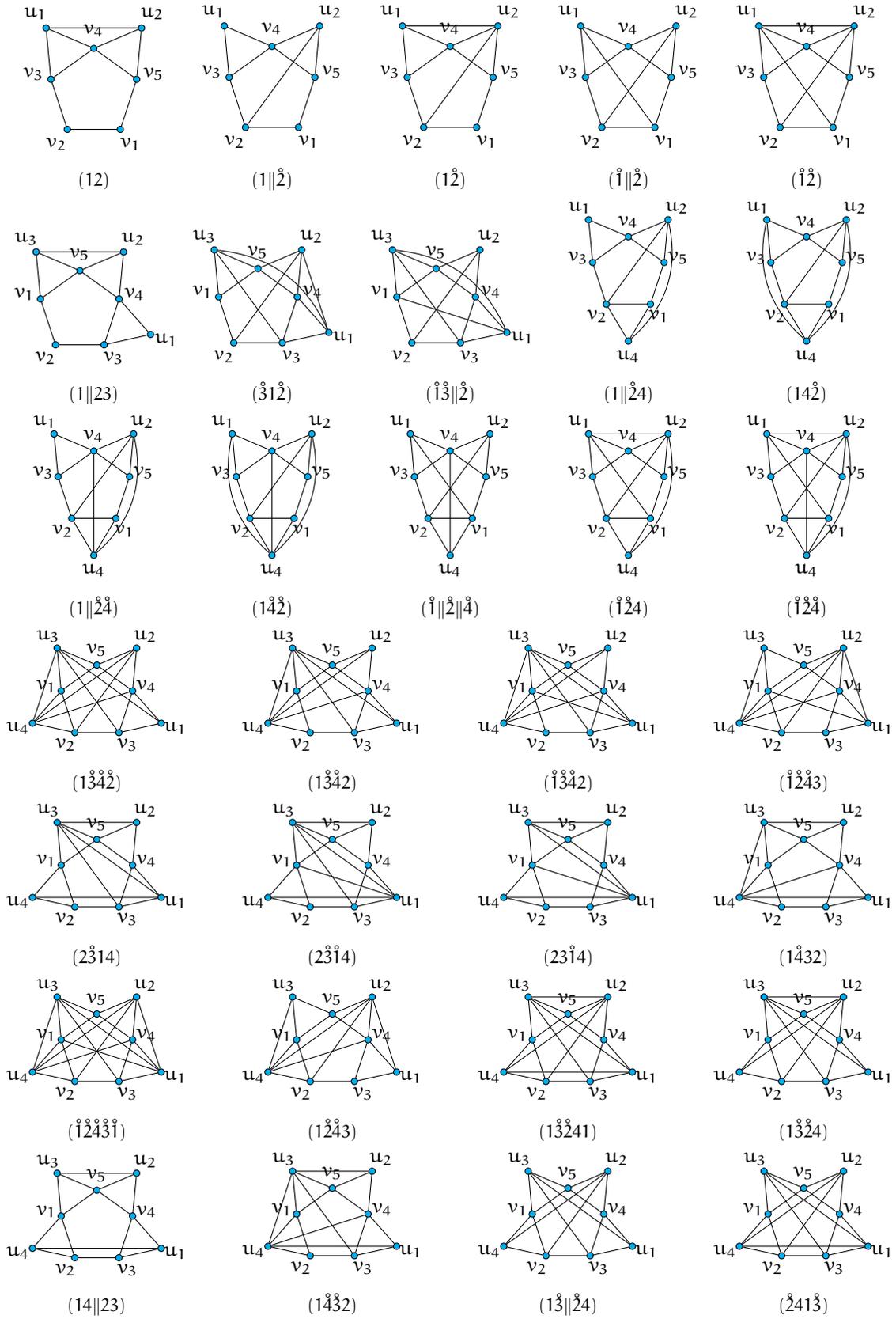
\begin{figure}[h]
  \centering
  \begin{subfigure}[b]{.18\textwidth}
    \centering
    \begin{tikzpicture} [scale=0.6]
      \foreach \i in {1,..., 5} {
          \draw ({18 - \i * (360 / 5)}:1.25) -- ({90 - \i * (360 / 5)}:1.25);
        }
      \foreach \i in {1,..., 5} {
          \node[filled vertex] (v\i) at ({18 - \i * (360 / 5)}:1.25) {};
          \node at ({18 - \i * (360 / 5)}:1.75) {$v_{\i}$};
        }
      \foreach \i in {1,..., 2} {
          \node[filled vertex] (u\i) at ({198 - \i * (360 / 5)}:2.25) {};
          \node at ({198 - \i * (360 / 5)}:2.75) {$u_{\i}$};
        }
      \draw (u1) -- (v3) (u1) -- (v4);
      \draw (u2) -- (v4) (u2) -- (v5);
      \draw (u1) -- (u2);
    \end{tikzpicture}
    \caption*{$(12)$}
  \end{subfigure}
  \begin{subfigure}[b]{.18\textwidth}
    \centering
    \begin{tikzpicture} [scale=0.6]
      \foreach \i in {1,..., 5} {
          \draw ({18 - \i * (360 / 5)}:1.25) -- ({90 - \i * (360 / 5)}:1.25);
        }
      \foreach \i in {1,..., 5} {
          \node[filled vertex] (v\i) at ({18 - \i * (360 / 5)}:1.25) {};
          \node at ({18 - \i * (360 / 5)}:1.75) {$v_{\i}$};
        }
      \foreach \i in {1,..., 2} {
          \node[filled vertex] (u\i) at ({198 - \i * (360 / 5)}:2.25) {};
          \node at ({198 - \i * (360 / 5)}:2.75) {$u_{\i}$};
        }
      \draw (u1) -- (v3) (u1) -- (v4);
      \draw (u2) -- (v4) (u2) -- (v5);
      \draw (u2) -- (v2);
    \end{tikzpicture}
    \caption*{$(1\|\mathring{2})$}
  \end{subfigure}
  \begin{subfigure}[b]{.18\textwidth}
    \centering
    \begin{tikzpicture} [scale=0.6]
      \foreach \i in {1,..., 5} {
          \draw ({18 - \i * (360 / 5)}:1.25) -- ({90 - \i * (360 / 5)}:1.25);
        }
      \foreach \i in {1,..., 5} {
          \node[filled vertex] (v\i) at ({18 - \i * (360 / 5)}:1.25) {};
          \node at ({18 - \i * (360 / 5)}:1.75) {$v_{\i}$};
        }
      \foreach \i in {1,..., 2} {
          \node[filled vertex] (u\i) at ({198 - \i * (360 / 5)}:2.25) {};
          \node at ({198 - \i * (360 / 5)}:2.75) {$u_{\i}$};
        }
      \draw (u1) -- (v3) (u1) -- (v4);
      \draw (u2) -- (v4) (u2) -- (v5);
      \draw (u2) -- (v2);
      \draw (u1) -- (u2);
    \end{tikzpicture}
    \caption*{$(1\mathring{2})$}
  \end{subfigure}
  \begin{subfigure}[b]{.18\textwidth}
    \centering
    \begin{tikzpicture} [scale=0.6]
      \foreach \i in {1,..., 5} {
          \draw ({18 - \i * (360 / 5)}:1.25) -- ({90 - \i * (360 / 5)}:1.25);
        }
      \foreach \i in {1,..., 5} {
          \node[filled vertex] (v\i) at ({18 - \i * (360 / 5)}:1.25) {};
          \node at ({18 - \i * (360 / 5)}:1.75) {$v_{\i}$};
        }
      \foreach \i in {1,..., 2} {
          \node[filled vertex] (u\i) at ({198 - \i * (360 / 5)}:2.25) {};
          \node at ({198 - \i * (360 / 5)}:2.75) {$u_{\i}$};
        }
      \draw (u1) -- (v3) (u1) -- (v4);
      \draw (u2) -- (v4) (u2) -- (v5);
      \draw (u1) -- (v1) (u2) -- (v2);
    \end{tikzpicture}
    \caption*{$(\mathring{1}\|\mathring{2})$}
  \end{subfigure}
  \begin{subfigure}[b]{.18\textwidth}
    \centering
    \begin{tikzpicture} [scale=0.6]
      \foreach \i in {1,..., 5} {
          \draw ({18 - \i * (360 / 5)}:1.25) -- ({90 - \i * (360 / 5)}:1.25);
        }
      \foreach \i in {1,..., 5} {
          \node[filled vertex] (v\i) at ({18 - \i * (360 / 5)}:1.25) {};
          \node at ({18 - \i * (360 / 5)}:1.75) {$v_{\i}$};
        }
      \foreach \i in {1,..., 2} {
          \node[filled vertex] (u\i) at ({198 - \i * (360 / 5)}:2.25) {};
          \node at ({198 - \i * (360 / 5)}:2.75) {$u_{\i}$};
        }
      \draw (u1) -- (v3) (u1) -- (v4);
      \draw (u2) -- (v4) (u2) -- (v5);
      \draw (u1) -- (v1) (u2) -- (v2);
      \draw (u1) -- (u2);
    \end{tikzpicture}
    \caption*{$(\mathring{1}\mathring{2})$}
  \end{subfigure}

  \begin{subfigure}[b]{.18\textwidth}
    \centering
    \begin{tikzpicture} [scale=0.55]
      \foreach \i in {1,..., 5} {
          \draw ({18 - \i * (360 / 5)}:1.25) -- ({90 - \i * (360 / 5)}:1.25);
        }
      \foreach \i in {1,..., 5} {
          \node[filled vertex] (v\i) at ({90 + \i * (360 / 5)}:1.25) {};
          \node at ({90 + \i * (360 / 5)}:1.75) {$v_{\i}$};
        }
      \foreach \i in {1,..., 3} {
          \node[filled vertex] (u\i) at ({-90 + \i * (360 / 5)}:2.25) {};
          \node at ({-90 + \i * (360 / 5)}:2.75) {$u_{\i}$};
        }
      \draw (u1) -- (v3) (u1) -- (v4);
      \draw (u2) -- (v4) (u2) -- (v5);
      \draw (u3) -- (v1) (u3) -- (v5);
      \draw (u2) -- (u3);
    \end{tikzpicture}
    \caption*{$(1\|23)$}
  \end{subfigure}
  \begin{subfigure}[b]{.18\textwidth}
    \centering
    \begin{tikzpicture} [scale=0.55]
      \foreach \i in {1,..., 5} {
          \draw ({18 - \i * (360 / 5)}:1.25) -- ({90 - \i * (360 / 5)}:1.25);
        }
      \foreach \i in {1,..., 5} {
          \node[filled vertex] (v\i) at ({90 + \i * (360 / 5)}:1.25) {};
          \node at ({90 + \i * (360 / 5)}:1.75) {$v_{\i}$};
        }
      \foreach \i in {1,..., 3} {
          \node[filled vertex] (u\i) at ({-90 + \i * (360 / 5)}:2.25) {};
          \node at ({-90 + \i * (360 / 5)}:2.75) {$u_{\i}$};
        }
      \draw (u1) -- (v3) (u1) -- (v4);
      \draw (u2) -- (v4) (u2) -- (v5);
      \draw (u3) -- (v1) (u3) -- (v5);
      \draw (u2) -- (v2) (u3) -- (v3);
      \draw (u1) edge [bend right] (u3);
      \draw (u1) -- (u2);
    \end{tikzpicture}
    \caption*{$(\mathring{3}1\mathring{2})$}
  \end{subfigure}
  \begin{subfigure}[b]{.18\textwidth}
    \centering
    \begin{tikzpicture} [scale=0.55]
      \foreach \i in {1,..., 5} {
          \draw ({18 - \i * (360 / 5)}:1.25) -- ({90 - \i * (360 / 5)}:1.25);
        }
      \foreach \i in {1,..., 5} {
          \node[filled vertex] (v\i) at ({90 + \i * (360 / 5)}:1.25) {};
          \node at ({90 + \i * (360 / 5)}:1.75) {$v_{\i}$};
        }
      \foreach \i in {1,..., 3} {
          \node[filled vertex] (u\i) at ({-90 + \i * (360 / 5)}:2.25) {};
          \node at ({-90 + \i * (360 / 5)}:2.75) {$u_{\i}$};
        }
      \draw (u1) -- (v3) (u1) -- (v4);
      \draw (u2) -- (v4) (u2) -- (v5);
      \draw (u3) -- (v1) (u3) -- (v5);
      \draw (u1) -- (v1) (u2) -- (v2) (u3) -- (v3);
      \draw (u1) edge [bend right] (u3);
    \end{tikzpicture}
    \caption*{$(\mathring{1}\mathring{3}\|\mathring{2})$}
  \end{subfigure}
  \begin{subfigure}[b]{.18\textwidth}
    \centering
    \begin{tikzpicture} [scale=0.5]
      \foreach \i in {1,..., 5} {
          \draw ({18 - \i * (360 / 5)}:1.25) -- ({90 - \i * (360 / 5)}:1.25);
        }
      \foreach \i in {1,..., 5} {
          \node[filled vertex] (v\i) at ({18 - \i * (360 / 5)}:1.25) {};
          \node at ({18 - \i * (360 / 5)}:1.75) {$v_{\i}$};
        }
      \node[filled vertex] (u1) at (126:2.25) {};
      \node[filled vertex] (u2) at (54:2.25) {};
      \node[filled vertex] (u4) at (270:2.25) {};
      \node at (126:2.75) {$u_{1}$};
      \node at (54:2.75) {$u_{2}$};
      \node at (270:2.75) {$u_{4}$};
      \draw (u1) -- (v3) (u1) -- (v4);
      \draw (u2) -- (v4) (u2) -- (v5);
      \draw (u4) -- (v1) (u4) -- (v2);
      \draw (u2) -- (v2);
      \draw (u2) edge [bend left] (u4);
    \end{tikzpicture}
    \caption*{$(1\|\mathring{2}4)$}
  \end{subfigure}
  \begin{subfigure}[b]{.18\textwidth}
    \centering
    \begin{tikzpicture} [scale=0.5]
      \foreach \i in {1,..., 5} {
          \draw ({18 - \i * (360 / 5)}:1.25) -- ({90 - \i * (360 / 5)}:1.25);
        }
      \foreach \i in {1,..., 5} {
          \node[filled vertex] (v\i) at ({18 - \i * (360 / 5)}:1.25) {};
          \node at ({18 - \i * (360 / 5)}:1.75) {$v_{\i}$};
        }
      \node[filled vertex] (u1) at (126:2.25) {};
      \node[filled vertex] (u2) at (54:2.25) {};
      \node[filled vertex] (u4) at (270:2.25) {};
      \node at (126:2.75) {$u_{1}$};
      \node at (54:2.75) {$u_{2}$};
      \node at (270:2.75) {$u_{4}$};
      \draw (u1) -- (v3) (u1) -- (v4);
      \draw (u2) -- (v4) (u2) -- (v5);
      \draw (u4) -- (v1) (u4) -- (v2);
      \draw (u2) -- (v2);
      \draw (u1) edge [bend right] (u4) (u2) edge [bend left] (u4);
    \end{tikzpicture}
    \caption*{$(14\mathring{2})$}
  \end{subfigure}

  \begin{subfigure}[b]{.18\textwidth}
    \centering
    \begin{tikzpicture} [scale=0.5]
      \foreach \i in {1,..., 5} {
          \draw ({18 - \i * (360 / 5)}:1.25) -- ({90 - \i * (360 / 5)}:1.25);
        }
      \foreach \i in {1,..., 5} {
          \node[filled vertex] (v\i) at ({18 - \i * (360 / 5)}:1.25) {};
          \node at ({18 - \i * (360 / 5)}:1.75) {$v_{\i}$};
        }
      \node[filled vertex] (u1) at (126:2.25) {};
      \node[filled vertex] (u2) at (54:2.25) {};
      \node[filled vertex] (u4) at (270:2.25) {};
      \node at (126:2.75) {$u_{1}$};
      \node at (54:2.75) {$u_{2}$};
      \node at (270:2.75) {$u_{4}$};
      \draw (u1) -- (v3) (u1) -- (v4);
      \draw (u2) -- (v4) (u2) -- (v5);
      \draw (u4) -- (v1) (u4) -- (v2);
      \draw (u2) -- (v2) (u4) -- (v4) (u2) edge [bend left] (u4);
    \end{tikzpicture}
    \caption*{$(1\|\mathring{2}\mathring{4})$}
  \end{subfigure}
  \begin{subfigure}[b]{.18\textwidth}
    \centering
    \begin{tikzpicture} [scale=0.5]
      \foreach \i in {1,..., 5} {
          \draw ({18 - \i * (360 / 5)}:1.25) -- ({90 - \i * (360 / 5)}:1.25);
        }
      \foreach \i in {1,..., 5} {
          \node[filled vertex] (v\i) at ({18 - \i * (360 / 5)}:1.25) {};
          \node at ({18 - \i * (360 / 5)}:1.75) {$v_{\i}$};
        }
      \node[filled vertex] (u1) at (126:2.25) {};
      \node[filled vertex] (u2) at (54:2.25) {};
      \node[filled vertex] (u4) at (270:2.25) {};
      \node at (126:2.75) {$u_{1}$};
      \node at (54:2.75) {$u_{2}$};
      \node at (270:2.75) {$u_{4}$};
      \draw (u1) -- (v3) (u1) -- (v4);
      \draw (u2) -- (v4) (u2) -- (v5);
      \draw (u4) -- (v1) (u4) -- (v2);
      \draw (u2) -- (v2) (u4) -- (v4)  (u1) edge [bend right] (u4) (u2) edge [bend left] (u4);
    \end{tikzpicture}
    \caption*{$(1\mathring{4}\mathring{2})$}
  \end{subfigure}
  \begin{subfigure}[b]{.18\textwidth}
    \centering
    \begin{tikzpicture} [scale=0.5]
      \foreach \i in {1,..., 5} {
          \draw ({18 - \i * (360 / 5)}:1.25) -- ({90 - \i * (360 / 5)}:1.25);
        }
      \foreach \i in {1,..., 5} {
          \node[filled vertex] (v\i) at ({18 - \i * (360 / 5)}:1.25) {};
          \node at ({18 - \i * (360 / 5)}:1.75) {$v_{\i}$};
        }
      \node[filled vertex] (u1) at (126:2.25) {};
      \node[filled vertex] (u2) at (54:2.25) {};
      \node[filled vertex] (u4) at (270:2.25) {};
      \node at (126:2.75) {$u_{1}$};
      \node at (54:2.75) {$u_{2}$};
      \node at (270:2.75) {$u_{4}$};
      \draw (u1) -- (v3) (u1) -- (v4);
      \draw (u2) -- (v4) (u2) -- (v5);
      \draw (u4) -- (v1) (u4) -- (v2);
      \draw (u1) -- (v1) (u2) -- (v2) (u4) -- (v4);
    \end{tikzpicture}
    \caption*{$(\mathring{1}\|\mathring{2}\|\mathring{4})$}
  \end{subfigure}
  \begin{subfigure}[b]{.18\textwidth}
    \centering
    \begin{tikzpicture} [scale=0.5]
      \foreach \i in {1,..., 5} {
          \draw ({18 - \i * (360 / 5)}:1.25) -- ({90 - \i * (360 / 5)}:1.25);
        }
      \foreach \i in {1,..., 5} {
          \node[filled vertex] (v\i) at ({18 - \i * (360 / 5)}:1.25) {};
          \node at ({18 - \i * (360 / 5)}:1.75) {$v_{\i}$};
        }
      \node[filled vertex] (u1) at (126:2.25) {};
      \node[filled vertex] (u2) at (54:2.25) {};
      \node[filled vertex] (u4) at (270:2.25) {};
      \node at (126:2.75) {$u_{1}$};
      \node at (54:2.75) {$u_{2}$};
      \node at (270:2.75) {$u_{4}$};
      \draw (u1) -- (v3) (u1) -- (v4);
      \draw (u2) -- (v4) (u2) -- (v5);
      \draw (u4) -- (v1) (u4) -- (v2);
      \draw (u1) -- (v1) (u2) -- (v2);
      \draw (u1) -- (u2) (u2) edge [bend left] (u4);
    \end{tikzpicture}
    \caption*{$(\mathring{1}\mathring{2}4)$}
  \end{subfigure}
  \begin{subfigure}[b]{.18\textwidth}
    \centering
    \begin{tikzpicture} [scale=0.5]
      \foreach \i in {1,..., 5} {
          \draw ({18 - \i * (360 / 5)}:1.25) -- ({90 - \i * (360 / 5)}:1.25);
        }
      \foreach \i in {1,..., 5} {
          \node[filled vertex] (v\i) at ({18 - \i * (360 / 5)}:1.25) {};
          \node at ({18 - \i * (360 / 5)}:1.75) {$v_{\i}$};
        }
      \node[filled vertex] (u1) at (126:2.25) {};
      \node[filled vertex] (u2) at (54:2.25) {};
      \node[filled vertex] (u4) at (270:2.25) {};
      \node at (126:2.75) {$u_{1}$};
      \node at (54:2.75) {$u_{2}$};
      \node at (270:2.75) {$u_{4}$};
      \draw (u1) -- (v3) (u1) -- (v4);
      \draw (u2) -- (v4) (u2) -- (v5);
      \draw (u4) -- (v1) (u4) -- (v2);
      \draw (u1) -- (v1) (u2) -- (v2) (u4) -- (v4);
      \draw (u1) -- (u2) (u2) edge [bend left] (u4);
    \end{tikzpicture}
    \caption*{$(\mathring{1}\mathring{2}\mathring{4})$}
  \end{subfigure}

  \begin{subfigure}[b]{.24\textwidth}
    \centering
    \begin{tikzpicture}[scale=0.5]
      \foreach \i in {1,..., 5} {
          \draw ({\i * (360 / 5) - 54}:1.25) -- ({\i * (360 / 5) + 18}:1.25);
        }
      \foreach \i in {1,..., 4} {
          \draw ({\i * (360 / 5) - 126}:1.25) -- ({\i * (360 / 5) - 90}:2.25) -- ({\i * (360 / 5) - 54}:1.25);
          \node[filled vertex] (u\i) at ({\i * (360 / 5) - 90}:2.25) {};
          \node at ({\i * (360 / 5) - 90}:2.75) {$u_{\i}$};
        }
      \foreach \i in {1,..., 5} {
          \node[filled vertex] (v\i) at ({\i * (360 / 5) + 90}:1.25) {};
          \node at ({\i * (360 / 5) + 90}:1.75) {$v_{\i}$};
        }
      \draw (u1) -- (u3) (u2) -- (u4);
      \draw (u3) -- (u4);
      \draw (u2) -- (v2) (u3) -- (v3) (u4) -- (v4);
    \end{tikzpicture}
    \caption*{$(1\mathring{3}\mathring{4}\mathring{2})$}
  \end{subfigure}
  \begin{subfigure}[b]{.24\textwidth}
    \centering
    \begin{tikzpicture}[scale=0.5]
      \foreach \i in {1,..., 5} {
          \draw ({\i * (360 / 5) - 54}:1.25) -- ({\i * (360 / 5) + 18}:1.25);
        }
      \foreach \i in {1,..., 4} {
          \draw ({\i * (360 / 5) - 126}:1.25) -- ({\i * (360 / 5) - 90}:2.25) -- ({\i * (360 / 5) - 54}:1.25);
          \node[filled vertex] (u\i) at ({\i * (360 / 5) - 90}:2.25) {};
          \node at ({\i * (360 / 5) - 90}:2.75) {$u_{\i}$};
        }
      \foreach \i in {1,..., 5} {
          \node[filled vertex] (v\i) at ({\i * (360 / 5) + 90}:1.25) {};
          \node at ({\i * (360 / 5) + 90}:1.75) {$v_{\i}$};
        }
      \draw (u1) -- (u3) (u2) -- (u4);
      \draw (u3) -- (u4);
      \draw (u3) -- (v3) (u4) -- (v4);
    \end{tikzpicture}
    \caption*{$(1\mathring{3}\mathring{4}2)$}
  \end{subfigure}
  \begin{subfigure}[b]{.24\textwidth}
    \centering
    \begin{tikzpicture}[scale=0.5]
      \foreach \i in {1,..., 5} {
          \draw ({\i * (360 / 5) - 54}:1.25) -- ({\i * (360 / 5) + 18}:1.25);
        }
      \foreach \i in {1,..., 4} {
          \draw ({\i * (360 / 5) - 126}:1.25) -- ({\i * (360 / 5) - 90}:2.25) -- ({\i * (360 / 5) - 54}:1.25);
          \node[filled vertex] (u\i) at ({\i * (360 / 5) - 90}:2.25) {};
          \node at ({\i * (360 / 5) - 90}:2.75) {$u_{\i}$};
        }
      \foreach \i in {1,..., 5} {
          \node[filled vertex] (v\i) at ({\i * (360 / 5) + 90}:1.25) {};
          \node at ({\i * (360 / 5) + 90}:1.75) {$v_{\i}$};
        }
      \draw (u1) -- (u3) (u2) -- (u4);
      \draw (u3) -- (u4);
      \draw (u1) -- (v1) (u3) -- (v3) (u4) -- (v4);
    \end{tikzpicture}
    \caption*{$(\mathring{1}\mathring{3}\mathring{4}2)$}
  \end{subfigure}
  \begin{subfigure}[b]{.24\textwidth}
    \centering
    \begin{tikzpicture}[scale=0.5]
      \foreach \i in {1,..., 5} {
          \draw ({\i * (360 / 5) - 54}:1.25) -- ({\i * (360 / 5) + 18}:1.25);
        }
      \foreach \i in {1,..., 4} {
          \draw ({\i * (360 / 5) - 126}:1.25) -- ({\i * (360 / 5) - 90}:2.25) -- ({\i * (360 / 5) - 54}:1.25);
          \node[filled vertex] (u\i) at ({\i * (360 / 5) - 90}:2.25) {};
          \node at ({\i * (360 / 5) - 90}:2.75) {$u_{\i}$};
        }
      \foreach \i in {1,..., 5} {
          \node[filled vertex] (v\i) at ({\i * (360 / 5) + 90}:1.25) {};
          \node at ({\i * (360 / 5) + 90}:1.75) {$v_{\i}$};
        }
      \draw (u2) -- (u4);
      \draw (u1) -- (u2) (u3) -- (u4);
      \draw (u1) -- (v1) (u2) -- (v2) (u4) -- (v4);
    \end{tikzpicture}
    \caption*{$(\mathring{1}\mathring{2}\mathring{4}3)$}
  \end{subfigure}

  \begin{subfigure}[b]{.24\textwidth}
    \centering
    \begin{tikzpicture}[scale=0.5]
      \foreach \i in {1,..., 5} {
          \draw ({\i * (360 / 5) - 54}:1.25) -- ({\i * (360 / 5) + 18}:1.25);
        }
      \foreach \i in {1,..., 4} {
          \draw ({\i * (360 / 5) - 126}:1.25) -- ({\i * (360 / 5) - 90}:2.25) -- ({\i * (360 / 5) - 54}:1.25);
          \node[filled vertex] (u\i) at ({\i * (360 / 5) - 90}:2.25) {};
          \node at ({\i * (360 / 5) - 90}:2.75) {$u_{\i}$};
        }
      \foreach \i in {1,..., 5} {
          \node[filled vertex] (v\i) at ({\i * (360 / 5) + 90}:1.25) {};
          \node at ({\i * (360 / 5) + 90}:1.75) {$v_{\i}$};
        }
      \draw (u1) -- (u3) (u1) -- (u4);
      \draw (u2) -- (u3);
      \draw (u3) -- (v3);
    \end{tikzpicture}
    \caption*{$(2\mathring{3}14)$}
  \end{subfigure}
  \begin{subfigure}[b]{.24\textwidth}
    \centering
    \begin{tikzpicture}[scale=0.5]
      \foreach \i in {1,..., 5} {
          \draw ({\i * (360 / 5) - 54}:1.25) -- ({\i * (360 / 5) + 18}:1.25);
        }
      \foreach \i in {1,..., 4} {
          \draw ({\i * (360 / 5) - 126}:1.25) -- ({\i * (360 / 5) - 90}:2.25) -- ({\i * (360 / 5) - 54}:1.25);
          \node[filled vertex] (u\i) at ({\i * (360 / 5) - 90}:2.25) {};
          \node at ({\i * (360 / 5) - 90}:2.75) {$u_{\i}$};
        }
      \foreach \i in {1,..., 5} {
          \node[filled vertex] (v\i) at ({\i * (360 / 5) + 90}:1.25) {};
          \node at ({\i * (360 / 5) + 90}:1.75) {$v_{\i}$};
        }
      \draw (u1) -- (u3) (u1) -- (u4);
      \draw (u2) -- (u3);
      \draw (u1) -- (v1) (u3) -- (v3);
    \end{tikzpicture}
    \caption*{$(2\mathring{3}\mathring{1}4)$}
  \end{subfigure}
  \begin{subfigure}[b]{.24\textwidth}
    \centering
    \begin{tikzpicture}[scale=0.5]
      \foreach \i in {1,..., 5} {
          \draw ({\i * (360 / 5) - 54}:1.25) -- ({\i * (360 / 5) + 18}:1.25);
        }
      \foreach \i in {1,..., 4} {
          \draw ({\i * (360 / 5) - 126}:1.25) -- ({\i * (360 / 5) - 90}:2.25) -- ({\i * (360 / 5) - 54}:1.25);
          \node[filled vertex] (u\i) at ({\i * (360 / 5) - 90}:2.25) {};
          \node at ({\i * (360 / 5) - 90}:2.75) {$u_{\i}$};
        }
      \foreach \i in {1,..., 5} {
          \node[filled vertex] (v\i) at ({\i * (360 / 5) + 90}:1.25) {};
          \node at ({\i * (360 / 5) + 90}:1.75) {$v_{\i}$};
        }
      \draw (u1) -- (u3) (u1) -- (u4);
      \draw (u2) -- (u3);
      \draw (u1) -- (v1);
    \end{tikzpicture}
    \caption*{$(23\mathring{1}4)$}
  \end{subfigure}
  \begin{subfigure}[b]{.24\textwidth}
    \centering
    \begin{tikzpicture}[scale=0.5]
      \foreach \i in {1,..., 5} {
          \draw ({\i * (360 / 5) - 54}:1.25) -- ({\i * (360 / 5) + 18}:1.25);
        }
      \foreach \i in {1,..., 4} {
          \draw ({\i * (360 / 5) - 126}:1.25) -- ({\i * (360 / 5) - 90}:2.25) -- ({\i * (360 / 5) - 54}:1.25);
          \node[filled vertex] (u\i) at ({\i * (360 / 5) - 90}:2.25) {};
          \node at ({\i * (360 / 5) - 90}:2.75) {$u_{\i}$};
        }
      \foreach \i in {1,..., 5} {
          \node[filled vertex] (v\i) at ({\i * (360 / 5) + 90}:1.25) {};
          \node at ({\i * (360 / 5) + 90}:1.75) {$v_{\i}$};
        }
      \draw (u1) -- (u4);
      \draw (u2) -- (u3) (u3) -- (u4) ;
      \draw (u4) -- (v4);
    \end{tikzpicture}
    \caption*{$(1\mathring{4}32)$}
  \end{subfigure}

  \begin{subfigure}[b]{.24\textwidth}
    \centering
    \begin{tikzpicture}[scale=0.5]
      \foreach \i in {1,..., 5} {
          \draw ({\i * (360 / 5) - 54}:1.25) -- ({\i * (360 / 5) + 18}:1.25);
        }
      \foreach \i in {1,..., 4} {
          \draw ({\i * (360 / 5) - 126}:1.25) -- ({\i * (360 / 5) - 90}:2.25) -- ({\i * (360 / 5) - 54}:1.25);
          \node[filled vertex] (u\i) at ({\i * (360 / 5) - 90}:2.25) {};
          \node at ({\i * (360 / 5) - 90}:2.75) {$u_{\i}$};
        }
      \foreach \i in {1,..., 5} {
          \node[filled vertex] (v\i) at ({\i * (360 / 5) + 90}:1.25) {};
          \node at ({\i * (360 / 5) + 90}:1.75) {$v_{\i}$};
        }
      \draw (u1) -- (u3) (u2) -- (u4);
      \draw (u1) -- (u2) (u3) -- (u4);
      \draw (u1) -- (v1) (u2) -- (v2) (u3) -- (v3) (u4) -- (v4);
    \end{tikzpicture}
    \caption*{$(\mathring{1}\mathring{2}\mathring{4}\mathring{3}\mathring{1})$}
  \end{subfigure}
  \begin{subfigure}[b]{.24\textwidth}
    \centering
    \begin{tikzpicture}[scale=0.5]
      \foreach \i in {1,..., 5} {
          \draw ({\i * (360 / 5) - 54}:1.25) -- ({\i * (360 / 5) + 18}:1.25);
        }
      \foreach \i in {1,..., 4} {
          \draw ({\i * (360 / 5) - 126}:1.25) -- ({\i * (360 / 5) - 90}:2.25) -- ({\i * (360 / 5) - 54}:1.25);
          \node[filled vertex] (u\i) at ({\i * (360 / 5) - 90}:2.25) {};
          \node at ({\i * (360 / 5) - 90}:2.75) {$u_{\i}$};
        }
      \foreach \i in {1,..., 5} {
          \node[filled vertex] (v\i) at ({\i * (360 / 5) + 90}:1.25) {};
          \node at ({\i * (360 / 5) + 90}:1.75) {$v_{\i}$};
        }
      \draw (u2) -- (u4);
      \draw (u1) -- (u2) (u3) -- (u4);
      \draw (u2) -- (v2) (u4) -- (v4);
    \end{tikzpicture}
    \caption*{$(1\mathring{2}\mathring{4}3)$}
  \end{subfigure}
  \begin{subfigure}[b]{.24\textwidth}
    \centering
    \begin{tikzpicture}[scale=0.5]
      \foreach \i in {1,..., 5} {
          \draw ({\i * (360 / 5) - 54}:1.25) -- ({\i * (360 / 5) + 18}:1.25);
        }
      \foreach \i in {1,..., 4} {
          \draw ({\i * (360 / 5) - 126}:1.25) -- ({\i * (360 / 5) - 90}:2.25) -- ({\i * (360 / 5) - 54}:1.25);
          \node[filled vertex] (u\i) at ({\i * (360 / 5) - 90}:2.25) {};
          \node at ({\i * (360 / 5) - 90}:2.75) {$u_{\i}$};
        }
      \foreach \i in {1,..., 5} {
          \node[filled vertex] (v\i) at ({\i * (360 / 5) + 90}:1.25) {};
          \node at ({\i * (360 / 5) + 90}:1.75) {$v_{\i}$};
        }
      \draw (u1) -- (u3) (u1) -- (u4) (u2) -- (u4);
      \draw (u2) -- (u3);
      \draw (u2) -- (v2) (u3) -- (v3);
    \end{tikzpicture}
    \caption*{$(1\mathring{3}\mathring{2}41)$}
  \end{subfigure}
  \begin{subfigure}[b]{.24\textwidth}
    \centering
    \begin{tikzpicture}[scale=0.5]
      \foreach \i in {1,..., 5} {
          \draw ({\i * (360 / 5) - 54}:1.25) -- ({\i * (360 / 5) + 18}:1.25);
        }
      \foreach \i in {1,..., 4} {
          \draw ({\i * (360 / 5) - 126}:1.25) -- ({\i * (360 / 5) - 90}:2.25) -- ({\i * (360 / 5) - 54}:1.25);
          \node[filled vertex] (u\i) at ({\i * (360 / 5) - 90}:2.25) {};
          \node at ({\i * (360 / 5) - 90}:2.75) {$u_{\i}$};
        }
      \foreach \i in {1,..., 5} {
          \node[filled vertex] (v\i) at ({\i * (360 / 5) + 90}:1.25) {};
          \node at ({\i * (360 / 5) + 90}:1.75) {$v_{\i}$};
        }
      \draw (u1) -- (u3) (u2) -- (u4);
      \draw (u2) -- (u3);
      \draw (u2) -- (v2) (u3) -- (v3);
    \end{tikzpicture}
    \caption*{$(1\mathring{3}\mathring{2}4)$}
  \end{subfigure}

  \begin{subfigure}[b]{.24\textwidth}
    \centering
    \begin{tikzpicture}[scale=0.5]
      \foreach \i in {1,..., 5} {
          \draw ({\i * (360 / 5) - 54}:1.25) -- ({\i * (360 / 5) + 18}:1.25);
        }
      \foreach \i in {1,..., 4} {
          \draw ({\i * (360 / 5) - 126}:1.25) -- ({\i * (360 / 5) - 90}:2.25) -- ({\i * (360 / 5) - 54}:1.25);
          \node[filled vertex] (u\i) at ({\i * (360 / 5) - 90}:2.25) {};
          \node at ({\i * (360 / 5) - 90}:2.75) {$u_{\i}$};
        }
      \foreach \i in {1,..., 5} {
          \node[filled vertex] (v\i) at ({\i * (360 / 5) + 90}:1.25) {};
          \node at ({\i * (360 / 5) + 90}:1.75) {$v_{\i}$};
        }
      \draw (u1) -- (u4);
      \draw (u2) -- (u3);
    \end{tikzpicture}
    \caption*{$(14\|23)$}
  \end{subfigure}
  \begin{subfigure}[b]{.24\textwidth}
    \centering
    \begin{tikzpicture}[scale=0.5]
      \foreach \i in {1,..., 5} {
          \draw ({\i * (360 / 5) - 54}:1.25) -- ({\i * (360 / 5) + 18}:1.25);
        }
      \foreach \i in {1,..., 4} {
          \draw ({\i * (360 / 5) - 126}:1.25) -- ({\i * (360 / 5) - 90}:2.25) -- ({\i * (360 / 5) - 54}:1.25);
          \node[filled vertex] (u\i) at ({\i * (360 / 5) - 90}:2.25) {};
          \node at ({\i * (360 / 5) - 90}:2.75) {$u_{\i}$};
        }
      \foreach \i in {1,..., 5} {
          \node[filled vertex] (v\i) at ({\i * (360 / 5) + 90}:1.25) {};
          \node at ({\i * (360 / 5) + 90}:1.75) {$v_{\i}$};
        }
      \draw (u1) -- (u4);
      \draw (u2) -- (u3) (u3) -- (u4);
      \draw (u3) -- (v3) (u4) -- (v4) ;
    \end{tikzpicture}
    \caption*{$(1\mathring{4}\mathring{3}2)$}
  \end{subfigure}
  \begin{subfigure}[b]{.24\textwidth}
    \centering
    \begin{tikzpicture}[scale=0.5]
      \foreach \i in {1,..., 5} {
          \draw ({\i * (360 / 5) - 54}:1.25) -- ({\i * (360 / 5) + 18}:1.25);
        }
      \foreach \i in {1,..., 4} {
          \draw ({\i * (360 / 5) - 126}:1.25) -- ({\i * (360 / 5) - 90}:2.25) -- ({\i * (360 / 5) - 54}:1.25);
          \node[filled vertex] (u\i) at ({\i * (360 / 5) - 90}:2.25) {};
          \node at ({\i * (360 / 5) - 90}:2.75) {$u_{\i}$};
        }
      \foreach \i in {1,..., 5} {
          \node[filled vertex] (v\i) at ({\i * (360 / 5) + 90}:1.25) {};
          \node at ({\i * (360 / 5) + 90}:1.75) {$v_{\i}$};
        }
      \draw (u1) -- (u3) (u2) -- (u4);
      \draw (u2) -- (v2) (u3) -- (v3);
    \end{tikzpicture}
    \caption*{$(1\mathring{3}\|\mathring{2}4)$}
  \end{subfigure}
  \begin{subfigure}[b]{.24\textwidth}
    \centering
    \begin{tikzpicture}[scale=0.5]
      \foreach \i in {1,..., 5} {
          \draw ({\i * (360 / 5) - 54}:1.25) -- ({\i * (360 / 5) + 18}:1.25);
        }
      \foreach \i in {1,..., 4} {
          \draw ({\i * (360 / 5) - 126}:1.25) -- ({\i * (360 / 5) - 90}:2.25) -- ({\i * (360 / 5) - 54}:1.25);
          \node[filled vertex] (u\i) at ({\i * (360 / 5) - 90}:2.25) {};
          \node at ({\i * (360 / 5) - 90}:2.75) {$u_{\i}$};
        }
      \foreach \i in {1,..., 5} {
          \node[filled vertex] (v\i) at ({\i * (360 / 5) + 90}:1.25) {};
          \node at ({\i * (360 / 5) + 90}:1.75) {$v_{\i}$};
        }
      \draw (u1) -- (u3) (u1) -- (u4) (u2) -- (u4);
      \draw (u2) -- (v2) (u3) -- (v3);
    \end{tikzpicture}
    \caption*{$(\mathring{2}41\mathring{3})$}
  \end{subfigure}
  \caption{Some t-perfect graphs}
  \label{fig:t-perfect graphs}
\end{figure}

\begin{table} [h]
  \centering
  \caption{For the proof of Proposition~\ref{lem:t-perfect-graphs}}
  \label{table:2}
  \begin{tabular}{ l | c c c c}
    \toprule
                                                                     & {$K = \{a, b, c\}$}       & $G - a$                                          & $G - b$             & $G - c$                                              \\
    \midrule
    $(1\|23)$                                                        & $\{v_{3}, v_{4}, u_{1}\}$ & $\star$                                          & $\star$             & $(12)$                                               \\[1ex]
    $(1\|\mathring{2}4)$                                             & $\{v_{1}, v_{2}, u_{4}\}$ & $\star$                                          & $\star$             & $(1\|\mathring{2})$                                  \\[1ex]
    $(14\mathring{2})$                                               & $\{v_{1}, v_{2}, u_{4}\}$ & $\star$                                          & $(1\|\mathring{2})$ & $(1\|\mathring{2})$                                  \\[1ex]
    $(1\|\mathring{2}\mathring{4})$                                  & $\{v_{1}, v_{2}, u_{4}\}$ & $\star$                                          & $\star$             & $(1\|\mathring{2})$                                  \\[1ex]
    $(1\mathring{4}\mathring{2})$                                    & $\{v_{1}, v_{2}, u_{4}\}$ & $\star$                                          & $\star$             & $(1\|\mathring{2})$                                  \\[1ex]
    $(\mathring{1}\|\mathring{2}\|\mathring{4})$                     & $\{v_{1}, v_{2}, u_{4}\}$ & $\star$                                          & $\star$             & $(\mathring{1}\|\mathring{2})$                       \\[1ex]
    $(\mathring{1}\mathring{2}4)$                                    & $\{v_{1}, v_{2}, u_{4}\}$ & $\star$                                          & $\star$             & $(\mathring{1}\mathring{2})$                         \\[1ex]
    $(\mathring{1}\mathring{2}\mathring{4})$                         & $\{v_{4}, v_{5}, u_{2}\}$ & $\star$                                          & $\star$             & $(\mathring{1}\|\mathring{2}\|\mathring{4} - u_{2})$ \\[1ex]
    $(\mathring{3}1\mathring{2})$                                    & $\{v_{1}, v_{5}, u_{3}\}$ & $(\mathring{1}\mathring{2}\mathring{4} - u_{1})$ & $(1\mathring{2})$   & $(1\mathring{2})$                                    \\[1ex]
    $(\mathring{1}\mathring{3}\|\mathring{2})$                       & $\{v_{4}, v_{5}, u_{2}\}$ & $(1\mathring{2})$                                & $(1\mathring{2})$   & $(\mathring{1}\mathring{2}\mathring{4} - u_{1})$     \\[1ex]
    $(1\mathring{3}\mathring{4}\mathring{2})$                        & $\{v_{4}, v_{5}, u_{2}\}$ & $(1\|\mathring{2}\mathring{4})$                  & $\star$             & $(\mathring{1}\mathring{2}4)$                        \\[1ex]
    $(1\mathring{3}\mathring{4}2)$                                   & $\{v_{4}, v_{5}, u_{2}\}$ & $\star$                                          & $\star$             & $(\mathring{1}\mathring{2}4)$                        \\[1ex]
    $(\mathring{1}\mathring{3}\mathring{4}2)$                        & $\{v_{4}, v_{5}, u_{2}\}$ & $\star$                                          & $\star$             & $(\mathring{1}\mathring{2}\mathring{4})$             \\[1ex]
    $(\mathring{1}\mathring{2}\mathring{4}3)$                        & $\{v_{1}, v_{5}, u_{3}\}$ & $\star$                                          & $\star$             & $(\mathring{1}\mathring{2}\mathring{4})$             \\[1ex]
    $(2\mathring{3}14)$                                              & $\{v_{4}, v_{5}, u_{2}\}$ & $\star$                                          & $\star$             & $(14\mathring{2})$                                   \\[1ex]
    $(2\mathring{3}\mathring{1}4)$                                   & $\{v_{4}, v_{5}, u_{2}\}$ & $\star$                                          & $\star$             & $(1\mathring{4}\mathring{2})$                        \\[1ex]
    $(23\mathring{1}4)$                                              & $\{v_{3}, v_{4}, u_{1}\}$ & $\star$                                          & $\star$             & $(1\|23)$                                            \\[1ex]
    $(1\mathring{4}32)$                                              & $\{v_{3}, v_{4}, u_{1}\}$ & $\star$                                          & $\star$             & $(1\mathring{2}3451) - \{u_{3}, u_{4}\}$                   \\[1ex]
    $(\mathring{1}\mathring{2}\mathring{4}\mathring{3}\mathring{1})$ & $\{v_{4}, v_{5}, u_{2}\}$ & $(\mathring{1}\mathring{2}\mathring{4})$         & $\star$             & $(\mathring{1}\mathring{2}\mathring{4})$             \\[1ex]
    $(1\mathring{2}\mathring{4}3)$                                   & $\{v_{3}, v_{4}, u_{1}\}$ & $\star$                                          & $\star$             & $(\mathring{1}\mathring{2}\mathring{4}3) - u_{1}$    \\[1ex]
    $(1\mathring{3}\mathring{2}41)$                                  & $\{v_{4}, v_{5}, u_{2}\}$ & $\star$                                          & $\star$             & $(14\mathring{2})$                                   \\[1ex]
    $(1\mathring{3}\mathring{2}4)$                                   & $\{v_{4}, v_{5}, u_{2}\}$ & $\star$                                          & $\star$             & $(1\|\mathring{2}4)$                                 \\[1ex]
    $(14\|23)$                                                       & $\{v_{3}, v_{4}, u_{1}\}$ & $(1\|23)$                                        & $\star$             & $(1\|23)$                                            \\[1ex]
    $(1\mathring{4}\mathring{3}2)$                                   & $\{v_{3}, v_{4}, u_{1}\}$ & $\star$                                          & $\star$             & $(123\mathring{4}\mathring{5}1) - \{u_{1}, u_{2}\}$                  \\[1ex]
    $(1\mathring{3}\|\mathring{2}4)$                                 & $\{v_{4}, v_{5}, u_{2}\}$ & $(1\|\mathring{2}4)$                             & $\star$             & $(1\|\mathring{2}4)$                                 \\[1ex]
    $(\mathring{2}41\mathring{3})$                                   & $\{v_{4}, v_{5}, u_{2}\}$ & $(14\mathring{2})$                               & $\star$             & $(14\mathring{2})$                                   \\[1ex]
    \bottomrule
  \end{tabular}
\end{table}

\begin{lemma}
  \label{lem:minimally-t-imperfect-graphs}
  Graphs $(1\mathring{2}43\mathring{5}1)$ and $(\mathring{1}\mathring{2}43\mathring{5}\mathring{1})$ are minimally t-imperfect.
\end{lemma}
\begin{proof}
  First, we show that neither of $(1\mathring{2}43\mathring{5}1)$ and $(\mathring{1}\mathring{2}43\mathring{5}\mathring{1})$ is t-perfect.  Let $G$ be either of them, and  %
let ${x} = (\frac{1}{3}, \ldots, \frac{1}{3})^{\mathsf{T}}$.  It is not difficult to use the definition to verify that $x$ is in $P(G)$.  On the other hand, $x$ is not in the independent set polytope of $G$ because the maximum independent sets for both graph have order three, while the sum of elements in $x$ is $10\over 3$.  Thus, the independent set polytope of $G$ is different from $P(G)$.

  We now argue that every proper t-minor of $G$ is t-perfect.  Since $N(v)$ is not an independent set for every $v \in V(G)$, it suffices to show that $G - v$ is t-perfect for every $v \in V(G)$.
  This can be further simplified by the following two observations.  First, $G - u_{2}\cong G - u_{5}$ and $G - u_{3}\cong G - u_{4}$.  Second, $G - u_{i}\cong G - v_{i}$ for $i = 1, \dots, 5$.
  It is easy to verify that both $(1\mathring{2}43\mathring{5}1) - u_{1}$ and $(\mathring{1}\mathring{2}43\mathring{5}\mathring{1}) - u_{1}$ are isomorphic to $(\mathring{1}32\mathring{4})$, while $(1\mathring{2}43\mathring{5}1) - u_{2}\cong (21\mathring{3}4)$ and $(\mathring{1}\mathring{2}43\mathring{5}\mathring{1}) - u_{2}\cong (21\mathring{3}\mathring{4})$ and $(1\mathring{2}43\mathring{5}1) - u_{3}\cong (1\mathring{4}3\mathring{2})$ and $(\mathring{1}\mathring{2}43\mathring{5}\mathring{1}) - u_{3}\cong (1\mathring{4}\mathring{3}\mathring{2})$.
To prove the t-perfection of the five graphs $(\mathring{1}32\mathring{4})$, $(21\mathring{3}4)$, $(21\mathring{3}\mathring{4})$, $(1\mathring{4}3\mathring{2})$, and $(1\mathring{4}\mathring{3}\mathring{2})$-perfect, we use samilar arguments as those in Proposition~\ref{lem:t-perfect-graphs}, and the details are listed in Table~\ref{table:3}.
\end{proof}

\begin{table} [h]
  \centering
  \caption{For the proof of Lemma~\ref{lem:minimally-t-imperfect-graphs}}
  \label{table:3}
  \begin{tabular}{ l | c c c c}
    \toprule
                                              & {$K = \{a, b, c\}$}       & $G - a$ & $G - b$ & $G - c$                                                                   \\
    \midrule
    $(\mathring{1}32\mathring{4})$            & $\{v_{3}, v_{4}, u_{1}\}$ & $\star$ & $\star$ & $(23\mathring{1}4) - u_{4}$                                               \\[1ex]
    $(21\mathring{3}4)$                       & $\{v_{1}, v_{2}, u_{4}\}$ & $\star$ & $\star$ & $(23\mathring{1}4) - u_{4}$                                               \\[1ex]
    $(21\mathring{3}\mathring{4})$            & $\{v_{1}, v_{2}, u_{4}\}$ & $\star$ & $\star$ & $(23\mathring{1}4) - u_{4}$                                               \\[1ex]
    $(1\mathring{4}3\mathring{2})$            & $\{v_{3}, v_{4}, u_{1}\}$ & $\star$ & $\star$ & $(\mathring{1}2\mathring{3}45\mathring{1}) - \{u_{4}, u_{5}\}$            \\[1ex]
    $(1\mathring{4}\mathring{3}\mathring{2})$ & $\{v_{3}, v_{4}, u_{1}\}$ & $\star$ & $\star$ & $(\mathring{1}\mathring{2}\mathring{3}45\mathring{1}) - \{u_{4}, u_{5}\}$ \\[1ex]
    \bottomrule
  \end{tabular}
\end{table}

\end{document}